\documentclass[11pt]{article}
\usepackage{latexsym,amsfonts,amssymb,amsmath,amsthm}

\usepackage{color,comment}

\begin{document}

\newcommand{ \bl}{\color{blue}}
\newcommand {\rd}{\color{red}}
\newcommand{ \bk}{\color{black}}
\newcommand{ \gr}{\color{OliveGreen}}
\newcommand{ \mg}{\color{RedViolet}}

\newcommand{\norm}[1]{||#1||}
\newcommand{\normo}[1]{|#1|}
\newcommand{\secn}[1]{\addtocounter{section}{1}\par\medskip\noindent
 {\large \bf \thesection. #1}\par\medskip\setcounter{equation}{0}}

\newcommand{\secs}[1]{\addtocounter {section}{1}\par\medskip\noindent
 {\large \bf  #1}\par\medskip\setcounter{equation}{0}}
\renewcommand{\theequation}{\thesection.\arabic{equation}}

\setlength{\baselineskip}{16pt}

\newtheorem{theorem}{Theorem}[section]
\newtheorem{lemma}{Lemma}[section]
\newtheorem{proposition}{Proposition}[section]
\newtheorem{definition}{Definition}[section]
\newtheorem{example}{Example}[section]
\newtheorem{corollary}{Corollary}[section]

\newtheorem{remark}{Remark}[section]

\numberwithin{equation}{section}

\def\p{\partial}
\def\I{\textit}
\def\R{\mathbb R}
\def\C{\mathbb C}
\def\u{\underline}
\def\l{\lambda}
\def\a{\alpha}
\def\O{\Omega}
\def\e{\epsilon}
\def\ls{\lambda^*}
\def\D{\displaystyle}
\def\wyx{ \frac{w(y,t)}{w(x,t)}}
\def\imp{\Rightarrow}
\def\tE{\tilde E}
\def\tX{\tilde X}
\def\tH{\tilde H}
\def\tu{\tilde u}
\def\d{\mathcal D}
\def\aa{\mathcal A}
\def\DH{\mathcal D(\tH)}
\def\bE{\bar E}
\def\bH{\bar H}
\def\M{\mathcal M}
\renewcommand{\labelenumi}{(\arabic{enumi})}

\def\disp{\displaystyle}
\def\undertex#1{$\underline{\hbox{#1}}$}
\def\card{\mathop{\hbox{card}}}
\def\sgn{\mathop{\hbox{sgn}}}
\def\exp{\mathop{\hbox{exp}}}
\def\OFP{(\Omega,{\cal F},\PP)}
\newcommand\JM{Mierczy\'nski}
\newcommand\RR{\ensuremath{\mathbb{R}}}
\newcommand\CC{\ensuremath{\mathbb{C}}}
\newcommand\QQ{\ensuremath{\mathbb{Q}}}
\newcommand\ZZ{\ensuremath{\mathbb{Z}}}
\newcommand\NN{\ensuremath{\mathbb{N}}}
\newcommand\PP{\ensuremath{\mathbb{P}}}
\newcommand\abs[1]{\ensuremath{\lvert#1\rvert}}

\newcommand\normf[1]{\ensuremath{\lVert#1\rVert_{f}}}
\newcommand\normfRb[1]{\ensuremath{\lVert#1\rVert_{f,R_b}}}
\newcommand\normfRbone[1]{\ensuremath{\lVert#1\rVert_{f, R_{b_1}}}}
\newcommand\normfRbtwo[1]{\ensuremath{\lVert#1\rVert_{f,R_{b_2}}}}
\newcommand\normtwo[1]{\ensuremath{\lVert#1\rVert_{2}}}
\newcommand\norminfty[1]{\ensuremath{\lVert#1\rVert_{\infty}}}
\newcommand{\ds}{\displaystyle}

\title{Uniqueness and stability of coexistence states in two species models with/without chemotaxis on bounded heterogeneous environments}
\author{
Tahir  Bachar Issa and Wenxian Shen\thanks{Partially supported by the NSF grant DMS--1645673} \\
Department of Mathematics and Statistics\\
Auburn University\\
Auburn University, AL 36849\\
U.S.A. }

\date{}
\maketitle

\noindent {\bf Abstract.}  The current paper is concerned with the asymptotic dynamics
 of two species competition systems with/without chemotaxis in heterogeneous media. In the previous work \cite{ITBWS17a}, we find
  conditions on the parameters in such  systems  for the persistence of the two species and the existence of positive coexistence states.
 In this paper, we  find  conditions on  the parameters   for the uniqueness and stability of positive coexistence states of such systems.
 The established results are  new even for the two species competition systems without chemotaxis but with space dependent coefficients.

\medskip


\medskip


\section{Introduction and the statements of the main results}
\label{S:intro}

  In the current paper, we consider the following two species parabolic-parabolic-elliptic system with heterogeneous Lotka-Volterra type competition terms,
 \begin{equation}
 \label{u-v-w-eq00}
\begin{cases}
u_t=d_1\Delta u-\chi_1\nabla\cdot (u \nabla w)+u\Big(a_0(t,x)-a_1(t,x)u-a_2(t,x)v\Big),\quad x\in \Omega\cr
v_t=d_2\Delta v-\chi_2\nabla \cdot(v \nabla w)+v\Big(b_0(t,x)-b_1(t,x)u-b_2(t,x)v\Big),\quad x\in \Omega\cr
0=d_3\Delta w+k u+lv-\lambda w,\quad x\in \Omega \cr
\frac{\p u}{\p n}=\frac{\p v}{\p n}=\frac{\p w}{\p n}=0,\quad x\in\p\Omega,
\end{cases}
 \end{equation}
where $\Omega \subset \mathbb{R}^n(n\geq 1)$ is a bounded domain with smooth boundary, $u(x,t)$ and $v(x,t)$ represent the population densities
of two  mobile species and $w(x,t)$ is the population density of some chemical substance,    $d_i$ ($i=1,2,3$) are positive constants, $\chi_1,\chi_2,k,l,\lambda$ are nonnegative constants, and $a_i(t,x)$ and $b_i(t,x)$ ($i=0,1,2$) are positive smooth functions.

{ Chemotaxis refers  the tendency of cells, bacteria, or   organisms to orient the direction of their movements toward the increasing or decreasing concentration of a signaling chemical substance. It has a crucial role in a wide range of biological phenomena such as immune system response, embryo development, tumor growth, etc. (see  \cite{ISM04}).  Recent studies describe also macroscopic process such as population dynamics or gravitational collapse, etc.,  in terms of chemotaxis  (see \cite{DAL1991}).  Because of its  crucial  role in the above mentioned process and others, chemotaxis  has attracted great attention in both   biological and mathematical communities since the  pioneering works \cite{KS1970, KS71}  by  Keller and Segel in 1970's, in which Keller and Segel  proposed   a celebrated mathematical model (K-S model) made up of two parabolic equations to describe chemotaxis.}

 {System \eqref{u-v-w-eq00} is a
Keller-Segel type model of chemotaxis, modeling  the population dynamics of two
competitive biological species attracted by the same nutrition subject to Lotka-
Volterra dynamics.
The term $-\chi_1\nabla\cdot (u \nabla w)$  with $\chi_1>0$  in the first equation of \eqref{u-v-w-eq00} reflects the influence of the chemical substance on
the movement of species $u$, and the term $-\chi_2\nabla\cdot (v \nabla w)$  with $\chi_2>0$  in the second equation of \eqref{u-v-w-eq00} reflects the influence of the chemical substance on
the movement of species $v$.  Note that \eqref{u-v-w-eq00} with $a_i(\cdot,\cdot)$ and $b_i(\cdot,\cdot)$ ($i=0,1,2)$ being constant
functions was
proposed by Tello and Winkler in \cite{TW12} to study the population dynamics of two
competitive species attracted by the same nutrition. In reality, the environments of many living organisms are spatially and temporally heterogeneous. It is then important both biologically and mathematically to
study the dynamics of the two species chemotaxis model \eqref{u-v-w-eq00} with $a_i(\cdot,\cdot)$ and $b_i(\cdot,\cdot)$ ($i=0,1,2)$ being
time and space dependent functions.
}

 In the absence of chemotaxis, that is, $\chi_1=\chi_2=0$, the dynamics of \eqref{u-v-w-eq00} is determined by the first two equations, that is, the following two species competition system,
\begin{equation}
\begin{cases}
\label{u-v-eq00}
u_t=d_1\Delta u+u\Big(a_0(t,x)-a_1(t,x)u-a_2(t,x)v\Big),\quad x\in \Omega\cr
v_t=d_2\Delta v+v\Big(b_0(t,x)-b_1(t,x)u-b_2(t,x)v\Big),\quad x\in \Omega\cr
\frac{\p u}{\p n}=\frac{\p v}{\p n}=0,\quad x\in\p\Omega.
\end{cases}
 \end{equation}
For biological reasons, we are only interested in nonnegative solutions of \eqref{u-v-w-eq00} and \eqref{u-v-eq00}.

{Consider \eqref{u-v-eq00}. It is well known that for any given initial time $t_0\in\RR$ and  nonnegative initial functions $u_0(x)$ and $v_0(x)$ in $C(\bar \Omega)$, the solution $(u(t,x),v(t,x))$
 of \eqref{u-v-eq00} with $(u(t_0,x), v(t_0,x))=(u_0(x),v_0(x))$, denoted by $(u(t,x;t_0,u_0,v_0),v(t,x;t_0,u_0,v_0))$,
  exists globally
(that is,  it exists for all
$t\ge t_0$).
Among central dynamical issues  in \eqref{u-v-eq00} are  persistence, coexistence, and extinction, which
 have been widely studied (see  \cite{Ahm}, \cite{FuMa97}, \cite{HeSh},  \cite{HeSh02}, etc.). For example, in \cite{Ahm},
  the author provided sufficient conditions for the convergence and ultimate bounds of spatially homogeneous solutions of \eqref{u-v-eq00} with $a_i(t,x)$ and
 $b_i(t,x)$ ($i=0,1,2,3$) being independent of $x$ and satisfying certain conditions, which implies the persistence.  In \cite{HeSh02},
 the authors provided sufficient conditions for the uniform persistence, coexistence, and extinction in\eqref{u-v-eq00} with $a_i(t,x)$ and $b_i(t,x)$ ($i=0,1,2$) being almost periodic in $t$. For example, it is proved in \cite{HeSh02} that uniform persistence  occurs in \eqref{u-v-eq00} if the coefficients satisfy
 \begin{equation}
 \label{stability-cond-1-eq1}
 a_{0,\inf}>\frac{a_{2,\sup}b_{0,\sup}}{b_{2,\inf}},\quad b_{0,\inf}>\frac{a_{0,\sup}b_{1,\sup}}{a_{1,\inf}}
 \end{equation}
 where
 $$
 a_{i,\inf}=\inf_{t\in\RR,x\in\bar \Omega}a_i(t,x),\quad a_{i,\sup}=\sup_{t\in\RR,x\in\bar \Omega}a_i(t,x),
 $$
 and $b_{i,\inf}$ and $b_{i,\sup}$ are defined similarly. Note that the occurrence of  persistence in \eqref{u-v-eq00} implies the existence of coexistence states.
   However, there is little study on the uniqueness and stability of coexistence states
 of \eqref{u-v-eq00} in the case that $a_i$ and $b_i$ ($i=0,1,2$) depend on $x$.  The uniqueness and stability of coexistence states of \eqref{u-v-eq00} proved in Corollary \ref{stability-cor} and  in Theorem \ref{thm-nonlinear-stability-001}(3) are new. It should be pointed out that, in the study of \eqref{u-v-eq00}, the so called competitive comparison principle plays
 an important role.}

 { Consider \eqref{u-v-w-eq00}. It is important to investigate the role of chemotaxis in determining the dynamical behavior of the
 solutions with nonnegative initial functions. In particular, it is important to investigate the following questions: whether the presence of chemotaxis affects the global existence
 of solutions with nonnegative initial functions, or whether  finite time blow-up occurs; how to
 identify the circumstances  under which
 persistence or extinction occurs;  in the case that persistence occurs, whether the system has coexistence states, and if so, whether the coexistence states are unique and stable; whether chemotaxis induces new solution patterns;
 etc. Note that chemotaxis induces several difficulties in the study of \eqref{u-v-w-eq00}, including the lack of the so called competitive comparison principle.}

Several authors have studied the issues mentioned in the above for system \eqref{u-v-w-eq00} with constant coefficients
(see \cite{TBJLMM16, ITBRS17,   NT13,   SBRWSa,  StTeWi, TW12, WaYaZh}).
{ For example, in \cite{TW12}, the authors studied the global stability of positive constant coexistence state under some assumption on
the coefficients.
In \cite{StTeWi},  the authors considered  the competitive
exclusion  under some complicated smallness assumptions on
the chemotaxis rates. In \cite{WaYaZh}, the authors obtained nonconstant positive coexistence states induced by chemotaxis
for parameters in certain region.
 In \cite{ITBRS17}, the authors considered a more general competitive-cooperative chemotaxis system with nonlocal terms logistic sources and proved both the phenomena of coexistence and of exclusion for parameters in some natural range. }

 However, there is little study on these asymptotic dynamical issues for \eqref{u-v-w-eq00}
with general time and space dependent coefficients. { Besides  the difficulties induced by the chemotaxis,
the time and space dependence of the coefficients induces additional difficulties in the study of \eqref{u-v-w-eq00}. For example,
in the constant coefficients case, a coexistence state of \eqref{u-v-eq00} is also a coexistence state of \eqref{u-v-w-eq00}
(in this case, \eqref{u-v-eq00} has at most one constant coexistence state), and hence persistence automatically implies
the existence of coexistence states.  However, a coexistence state of \eqref{u-v-eq00} may not be a coexistence state of \eqref{u-v-w-eq00} when the coefficients are space dependent. Therefore,
to study the dynamics of \eqref{u-v-w-eq00} with time and space dependent coefficients, new techniques need to
be developed.}

In the very recent paper \cite{ITBWS17a}, we studied the global existence, persistence, existence of coexistence states, and extinction in \eqref{u-v-w-eq00}.
{ In particular, we found various parameter regions for the global existence, persistence, existence of coexistence states, and extinction in \eqref{u-v-w-eq00} (see Theorems  \ref{thm-global-000}-\ref{thm-entire-002} in the following for some of the results proved in \cite{ITBWS17a}).}
The objective of the current paper is to find parameter regions for the uniqueness and stability of coexistence states in \eqref{u-v-w-eq00}.
Observe that, even for \eqref{u-v-eq00},  the uniqueness and stability of coexistence states  has been studied only for some special cases.
 We obtain some new results about the uniqueness and stability of coexistence states in  \eqref{u-v-eq00} {(see Remark \ref{rmk-stability}(1) and Corollary \ref{stability-cor})}.
 { Observe also that the results established in this paper provide  various conditions  under which \eqref{u-v-w-eq00} has a unique stable coexistence state.   All the conditions depend on the chemotaxis sensitivity coefficients $\chi_1$ and $\chi_2$, which reflect some effects  of chemotaxis on the uniqueness and stability of coexistence states. There are still several important problems to be studied, for example,
 whether finite time blow-up occurs in \eqref{u-v-w-eq00} when $\chi_1$ and $\chi_2$ are not small;
 whether  chemotaxis makes species easier to persist or go extinct; what patterns may be induced by  chemotaxis when the coefficients are nonconstant; etc.  We plan to study  these interesting problem in our future works.}
 The reader is referred to \cite{ITBWS16} and to \cite{SBRWSb,SBRWSc} for the existing works for one species chemotaxis models with general time and space dependent coefficients on bounded domains and on unbounded domains, respectively.

In the following, we state the main results of the current paper. To do so,
  we first introduce some notations, assumptions,  and recall some results obtained in \cite{ITBWS17a}.  Let
 $$ {  C^+(\bar{\Omega})=\left\{u\in C(\bar{\Omega})  \,|\, u \geq 0 \right \}.}
   $$
 The following two assumptions are introduced in \cite{ITBWS17a} for the global existence of classical solutions of \eqref{u-v-w-eq00} with given positive initial functions.

 \smallskip
\noindent {\bf (H1)} {\it $a_i(t,x)$, $b_i(t,x)$, $\chi_i$  and  $d_3$, $k$ and $l$ satisfy
\begin{equation}
\label{global-existence-cond-eq00}
a_{1,\inf}>\frac{ k\chi_1}{d_3},\quad a_{2,\inf}\geq \frac{l \chi_1}{d_3}, \quad b_{1,\inf}\geq \frac{k \chi_2}{d_3},\quad \text{and} \quad b_{2,\inf}>\frac{ l\chi_2}{d_3}.
 \end{equation}
}

\noindent {\bf (H2)} {\it $a_i(t,x)$, $b_i(t,x)$, $\chi_i$  and  $d_3$, $k$ and $l$ satisfy
 \begin{equation}
\label{global-existence-cond-eq01}
a_{1,\inf}>\frac{ k\chi_1}{d_3},\quad  b_{2,\inf}>\frac{l \chi_2}{d_3},\quad {\rm and}\quad
\big(a_{1,\inf}-\frac{ k\chi_1}{d_3}\big)\big( b_{2,\inf}-\frac{ l\chi_2}{d_3}\big)>\frac{k \chi_2}{d_3}\frac{l \chi_1}{d_3}.
 \end{equation}
 }

It is proved  in \cite{ITBWS17a} that

\begin{theorem}{ (Global Existence)}
\label{thm-global-000}
\begin{itemize}
\item[(1)] Assume that { (H1)} holds. Then for any $t_0\in\RR$ and  $u_0,v_0 \in { C^+(\bar{\Omega})},$
$\eqref{u-v-w-eq00}$ has a unique  bounded  global  classical solution $(u(x,t;t_0,u_0,v_0),v(x,t;t_0,u_0,v_0)$, $w(x,t;t_0,u_0,v_0))$  which satisfies that
\begin{equation}\label{global-existence-eq00}
\lim_{t\to t_0+}\|u(\cdot,t;t_0,u_0,v_0)-u_0(\cdot)\|_{C^0(\bar\Omega)}+\|v(\cdot,t;t_0,u_0,v_0)-v_0(\cdot)\|_{C^0(\bar\Omega)}=0.
\end{equation}
 Moreover,  for any $\epsilon>0$, there is {$T(u_0,v_0,\epsilon)\ge 0$}  such that
\[0\leq u(x,t;t_0,u_0,v_0) \leq \bar A_1+\epsilon  \]
and
\[0\leq v(x,t;t_0,u_0,v_0)\leq \bar A_2+\epsilon  \]
for all $t\ge t_0+T(u_0,v_0,\epsilon)$, where
\begin{equation}
\label{I1-I2-overbar}
\bar A_1= \frac{a_{0,\sup}}{a_{1,\inf}-\frac{k\chi_1}{d_3}},\quad \bar A_2=\frac{b_{0,\sup}}{b_{2,\inf}-\frac{l\chi_2}{d_3}}.
\end{equation}
If $u_0\le \bar A_1+\epsilon$, $v_0\le \bar A_2+\epsilon$, then $T(u_0,v_0,\epsilon)$ can be chosen to be zero.

\item[(2)] Assume that { (H2)} holds. Then for any $t_0\in\RR$ and   $u_0,v_0 \in { C^+(\bar{\Omega})}$,
$\eqref{u-v-w-eq00}$ has a unique  bounded  global  classical solution $(u(x,t;t_0,u_0,v_0),v(x,t;t_0,u_0,v_0)$, $w(x,t;t_0,u_0,v_0))$  which satisfies \eqref{global-existence-eq00}.  Moreover,  for any $\epsilon>0$, there is  $T(u_0,v_0,\epsilon)>0$ such that
$$0\leq u(x,t;t_0,u_0,v_0) \leq \bar B_1+\epsilon
$$
and
$$ 0\leq v(x,t;t_0,u_0,v_0)\leq \bar B_2+\epsilon
$$
for all $t\ge t_0+T(u_0,v_0,\epsilon)$, where
 \begin{equation}
 \label{A1-overbar-0}\bar {B_1}=\frac{a_{0,\sup}(b_{2,\inf}-\frac{l\chi_2}{d_3})+\frac{l\chi_1}{d_3}b_{0,\sup}}{(a_{1,\inf}-\frac{k\chi_1}{d_3})(b_{2,\inf}-\frac{l\chi_2}{d_3})
 -\frac{lk\chi_1\chi_2}{d_3^2}}
 \end{equation}
and
\begin{equation}
\label{A2-overbar-0}
\bar {B_2}=\frac{b_{0,\sup}(a_{1,\inf}-\frac{k\chi_1}{d_3})+\frac{k\chi_2}{d_3}a_{0,\sup}}{(a_{1,\inf}-\frac{k\chi_1}{d_3})(b_{2,\inf}
-\frac{l\chi_2}{d_3})-\frac{lk\chi_1\chi_2}{d_3^2}}.
\end{equation}
If $u_0\le \bar B_1+\epsilon$, $v_0\le \bar B_2+\epsilon$, $T(u_0,v_0,\epsilon)$ can be chosen to be zero.
\end{itemize}
\end{theorem}

The reader is referred to \cite[Remark 1.1]{ITBWS17a} for some remarks  about Theorem \ref{thm-global-000}, and the assumptions (H1) and (H2).
{ The following concepts of coexistence state and persistence are introduced in \cite{ITBWS17a}.}

\begin{definition}
\label{coex-persist-def}
A solution $(u(x,t),v(x,t),w(x,t))$ of \eqref{u-v-w-eq00} defined for all $t\in\RR$ is called an {\rm entire solution}.
A {\rm coexistence state} of \eqref{u-v-w-eq00} is an entire positive solution $(u^{**}(x,t),v^{**}(x,t),w^{**}(x,t))$ with
\begin{equation*}
\label{coexistence-eq}
\inf_{t\in\RR,x\in\bar\Omega} u^{**}(x,t)>0,\quad \inf_{t\in\RR,x\in\bar\Omega} v^{**}(x,t)>0.
\end{equation*}
We say that {\rm persistence} occurs in \eqref{u-v-w-eq00} if there is $\eta>0$ such that for any $t_0\in\RR$ and $u_0,v_0\in C(\bar\Omega)$ with $u_0> 0$ and  $v_0> 0$, there is $\tau(t_0,u_0,v_0)>0$ such that
\begin{equation*}
\label{persistence-eq}
u(x,t;t_0,u_0,v_0)\ge \eta,\quad v(x,t;t_0,u_0,v_0)\ge \eta\quad \forall\,\, t\ge t_0+\tau(t_0,u_0,v_0).
\end{equation*}
\end{definition}

{ The  following two assumptions are introduced in  \cite{ITBWS17a} for the persistence in \eqref{u-v-w-eq00}. }

\medskip
\noindent {\bf (H3)} {\it $a_i(t,x)$, $b_i(t,x)$, $\chi_i$ and  $d_3$, $k$ and $l$ satisfy (H1) and
 \begin{equation}
\label{invariant-rectangle-cond-eq-1}
a_{0,\inf}>a_{2,\sup}\bar A_2\,\,\,\,\,\text{and} \,\,  \,\,\, b_{0,\inf}>b_{1,\sup}\bar A_1.
 \end{equation}}

\noindent {\bf (H4)} {\it $a_i(t,x)$, $b_i(t,x)$, $\chi_i$ and  $d_3$, $k$ and $l$ satisfy (H2) and
 \begin{equation}
\label{invariant-rectangle-cond-eq-2}
a_{0,\inf}>\big(a_{2,\sup}-\frac{\chi_1 l}{d_3}\big)_+\bar B_2+\frac{\chi_1 l}{d_3}\bar B_2\,\,\, \,\,\text{and} \,\,\,\,\,  b_{0,\inf}> \big(b_{1,\sup}
-\frac{\chi_2 k}{d_3}\big)_+\bar B_1+\frac{\chi_2 k}{d_3}\bar B_1,
 \end{equation}
 where $(\cdots)_+$ represents the positive part of the expression inside the brackets.
}

Note that both (H3) and (H4) imply \eqref{stability-cond-1-eq1}. As it is mentioned in the above, \eqref{stability-cond-1-eq1}
are sufficient conditions for the persistence in \eqref{u-v-eq00} to occur (see \cite[Theorem B]{HeSh02}).
The following theorems  on persistence and the existence of coexistence states of  \eqref{u-v-w-eq00} are proved in \cite{ITBWS17a}.

\begin{theorem} [Persistence]
\label{thm-entire-001}
\begin{itemize}
\item[(1)]
 Assume (H3). Then there are $\underbar A_1>0$ and $\underbar A_2>0$ such that for  any $\epsilon>0$ and  $u_0,v_0\in { C^+(\bar{\Omega})}$ with $u_0,v_0\not \equiv 0$, there exists $t_{\epsilon,u_0,v_0}$ such that
\begin{equation}\label{attracting-set-eq000}
\underbar A_1 \le u(x,t;t_0,u_0,v_0) \le \bar{A_1}+\epsilon,\,\,\,\underbar A_2 \le v(x,t;t_0,u_0,v_0) \le \bar{A_2}+\epsilon
\end{equation}
for all $x\in\bar\Omega$, { $t_0\in\RR$, and $t\ge t_0+t_{\epsilon,u_0,v_0}$}.

\item[(2)]
 Assume (H4). Then there are $\underbar B_1>0$ and $\underbar B_2>0$ such that  for  any $\epsilon>0$  and $u_0,v_0\in { C^+(\bar{\Omega})}$ with $u_0,v_0\not \equiv 0$, there exists $t_{\epsilon,u_0,v_0}$ such {  \eqref{attracting-set-eq000} } holds with $\underbar A_1$,
  $\bar A_1$, $\underbar A_2$, and $\bar A_2$ being replaced by $\underbar B_1$, $\bar B_1$,   $\underbar B_2$, and  $\bar B_2$, respectively.
\end{itemize}
\end{theorem}

\begin{theorem} [Coexistence]
\label{thm-entire-002}
\begin{itemize}
\item[(1)]
 Assume (H3). Then
there is a coexistence state $(u^{**}(x,t)$, $v^{**}(x,t)$, $ w^{**}(x,t))$ of \eqref{u-v-w-eq00}.
Moreover, the following hold.
\begin{itemize}
\item[(i)] If there is $T>0$ such that $a_i(t+T,x)=a_i(t,x),$ $b_i(t+T,x)=b_i(t,x)$ for $i=0,1,2$, then \eqref{u-v-w-eq00} has a $T$-periodic coexistence state $(u^{**}(x,t),v^{**}(x,t), w^{**}(x,t))$, that is,
 $$(u^{**}(x,t+T), v^{**}(x,t+T), w^{**}(x,t+T))=(u^{**}(x,t), v^{**}(x,t), w^{**}(x,t)).
 $$

\item[(ii)] If   $a_i(t,x)\equiv a_i(x),$ $b_i(t,x)\equiv b_i(x)$ for $i=0,1,2$, then  \eqref{u-v-w-eq00} has a steady state coexistence state
 $$(u^{**}(t,x), v^{**}(t,x),w^{**}(t,x))\equiv (u^{**}(x), v^{**}(x),w^{**}(x)).
 $$

\item[(iii)]  If  $a_i(t,x)\equiv a_i(t),$ $b_i(t,x)=b_i(t)$ for $i=0,1,2$,   then \eqref{u-v-w-eq00} has a  spatially homogeneous coexistence state
$$(u^{**}(x,t),v^{**}(x,t), w^{**}(x,t))\equiv (u^{**}(t),v^{**}(t), w^{**}(t))$$
 with $w^{**}(t)=ku^{**}(t)+lv^{**}(t)$, and if $a_i(t),$ $b_i(t)$ $(i=0,1,2)$ are periodic or almost periodic, so is   $(u^{**}(t),v^{**}(t),w^{**}(t))$.
\end{itemize}
\item[(2)]
 Assume (H4). Then there is a coexistence state $(u^{**}(x,t),v^{**}(x,t), w^{**}(x,t))$ of \eqref{u-v-w-eq00} { which satisfies (i)-(iii) of (1).}
\end{itemize}
\end{theorem}

The reader is referred to \cite[Remark 1.2]{ITBWS17a}
for some remarks about the assumptions (H3), (H4), and Theorems \ref{thm-entire-001} and \ref{thm-entire-002}.

 We now  state the  results of the current paper on the  stability and uniqueness of coexistence states in \eqref{u-v-w-eq00}.
 For convenience, we first introduce the following assumptions.

 \medskip

 \noindent {\bf (H5)} {\it  Assume (H1) and
\begin{equation}
\label{aux-stability-eq1}
a_{0,\inf}> a_{2,\sup}\bar A_2+k\frac{\chi_1 }{d_3}\bar A_1,\quad b_{0,\inf}>b_{1,\sup}\bar A_1+l\frac{\chi_2 }{d_3}\bar A_2.
\end{equation}}

 \noindent {\bf (H6)} {\it Assume (H2) and
\begin{equation}
\label{aux-stability-eq2}
a_{0,\inf}>{ (a_{2,\sup}+l\frac{\chi_1}{d_3})}\bar B_2+k\frac{\chi_1}{d_3}\bar B_1,\quad b_{0,\inf}>{ (b_{1,\sup}+k\frac{\chi_2}{d_3})}\bar B_1+l\frac{\chi_2}{d_3}\bar B_2.
\end{equation}}

 \noindent {\bf (H7)} {\it $a_i(t,x)\equiv a_i(t)$ and $b_i(t,x)\equiv b_i(t)$ ($i=0,1,2$) satisfy \eqref{stability-cond-1-eq1}
and
\begin{equation}
\label{stability-cond-1-eq2}
\inf_t \big\{a_1(t)-b_1(t)\big\} >2{ \frac{k}{d_3}(\chi_1+\chi_2)},
\quad
\inf_t \big\{b_2(t)-a_2(t)\big\} >2{\frac{l}{d_3}(\chi_1+\chi_2)}.
\end{equation}}

\begin{remark}
\label{remarks-on-rk}
\begin{itemize}
\item[(1)]  (H5) implies   (H3)  and (H6) implies (H4).

\item[(2)]   When $\chi_1=\chi_2=0$, (H5) and (H6) are the same, and
both  \eqref{aux-stability-eq1}  and \eqref{aux-stability-eq2}  become \eqref{stability-cond-1-eq1}.

\item[(3)] { When $k=l=d_3=1,$ $a_0=a_1=\mu_1,a_2=\mu_1\tilde{a_1},$ $b_0=b_2=\mu_1,$ and $b_1=\mu_2\tilde{a_2}$, (H7) become $2(\chi_1+\chi_2)+\mu_2\tilde{a_2}<\mu_1$ and $2(\chi_1+\chi_2)+\mu_1\tilde{a_1}<\mu_2.$ Thus  our result in Theorem \ref{thm-nonlinear-stability-001}(3)  recovers the stability result in \cite{TW12}.}

\end{itemize}
\end{remark}

 By Theorem \ref{thm-entire-001}, assuming
 (H3) (resp.,  (H4)), for any $\epsilon>0$, $[\underbar A_1,\bar A_1+\epsilon]\times [\underbar A_2,\bar A_2+\epsilon]$
 (resp., $[\underbar B_1,\bar B_1+\epsilon]\times [\underbar B_2,\bar B_2+\epsilon]$) is an attracting rectangle for \eqref{u-v-w-eq00}.
 The first main theorem of the current paper is about optimal attracting rectangles for \eqref{u-v-w-eq00} under the assumption (H5) (resp., (H6)).

 \begin{theorem}[Optimal attracting rectangle]
 \label{attracting-rectangle-thm} For given $u_0,v_0 \in { C^+(\bar{\Omega})},$ let $\overline{u}_0=\max_{x \in \bar \Omega}u_0(x),$ $\underline{u}_0=\min_{x \in \bar \Omega}u_0(x),$ $\overline{v}_0=\max_{x \in \bar \Omega}v_0(x)$ , $\underline{v}_0=\min_{x \in \bar \Omega}v_0(x)$.

 \begin{itemize}
 \item[(1)] Assume (H5) and that the following system has a unique solution $(\bar r_1,\bar r_2,\underbar r_1,\underbar r_2)$
 \begin{equation}
 \label{attracting-rectabgle-eq1}
 \begin{cases}
    (a_{1,\inf}-k\frac{\chi_1}{d_3}) \bar r_1={a_{0,\sup}-a_{2,\inf}\underbar  r_2-k\frac{\chi_1}{d_3}\underbar  r_1}\cr
   (b_{2,\inf}-l\frac{\chi_2}{d_3}) \bar r_2={b_{0,\sup}-b_{1,\inf}\underbar  r_1-k\frac{\chi_1}{d_3}\underbar  r_2}\cr
 (a_{1,\sup}-k\frac{\chi_1}{d_3})\underbar  r_1={a_{0,\inf}-a_{2,\sup}\bar  r_2-k\frac{\chi_1}{d_3}\bar  r_1}\cr
(b_{2,\sup}-l\frac{\chi_2}{d_3})\underbar  r_2={b_{0,\inf}-b_{1,\sup}\bar  r_1-l\frac{\chi_2}{d_3}\bar  r_2}.
\end{cases}
 \end{equation}
 Then $0<\underbar r_1\le \bar r_1$, $0<\underbar r_2 \le \bar r_2$, and
   for any $\epsilon>0$, $t_0\in\mathbb{R}$, and  $u_0,v_0\in C^0(\bar \Omega)$ with $\inf u_0>0$, $\inf v_0> 0,$   there exists $ t_{\epsilon,\overline u_0,\overline v_0,\underline u_0,\underline v_0}$ such that
\begin{equation}\label{attracting-set-eq00}
\begin{cases}
0<\underbar r_1{ -\epsilon} \le u(x,t;t_0,u_0,v_0) \le \bar r_1+\epsilon\cr
0<\underbar r_2{ -\epsilon} \le v(x,t;t_0,u_0,v_0) \le \bar r_2+\epsilon,
\end{cases}
\end{equation}
for all $x\in\bar\Omega$ { and $t\ge t_0+t_{\epsilon,\overline u_0,\overline v_0,\underline u_0,\underline v_0}$}. Furthermore
\begin{equation}\label{invariant-set-eq00}
  \underbar r_1\leq u_0\leq \bar r_1 \,\,\text{and}\,\, \underbar r_2\leq v_0\leq \bar r_2
\end{equation}
implies
\begin{equation}\label{invariant-set-eq01}
  \underbar r_1\leq u(x,t;t_0,u_0,v_0)\leq \bar r_1 \,\,\text{and}\,\, \underbar r_2\leq v(x,t;t_0,u_0,v_0)\leq \bar r_2 \,\, \forall t\geq t_0.
\end{equation}

 \item[(2)] Assume (H6) and that there is a unique solution $(\bar s_1,\bar s_1,\underbar s_1,\underbar s_2)$ of the following system,
 \begin{equation}
 \label{attracrting-rectangle-eq2}
 \begin{cases}
 \bar s_1=\frac{\left(a_{0,\sup}-{ (a_{2,\inf}+l\frac{\chi_1}{d_3})}{\underbar s_2}-k\frac{\chi_1}{d_3}{\underbar s_1}\right)(b_{2,\inf}-l\frac{\chi_2}{d_3})}{(a_{1,\inf}-k\frac{\chi_1}{d_3})(b_{2,\inf}-l\frac{\chi_2}{d_3})-lk\frac{\chi_1\chi_2}
{d_3^2}}+\frac{\frac{l\chi_1}{d_3}\left(b_{0,\sup}-{ (b_{1,\inf}+k\frac{\chi_2}{d_3}) }{\underbar s_1}-l\frac{\chi_2}{d_3}{\underbar s_2}\right)}{(a_{1,\inf}-k\frac{\chi_1}{d_3})(b_{2,\inf}-l\frac{\chi_2}{d_3})-lk\frac{\chi_1\chi_2}{d_3^2}}\cr
\bar s_2=\frac{\left(b_{0,\sup}-{ (b_{1,\inf}+k\frac{\chi_2}{d_3})}{\underbar s_1}-l\frac{\chi_2}{d_3}{\underbar s_2}\right)(a_{1,\inf}-k\frac{\chi_1}{d_3})} {(a_{1,\inf}-k\frac{\chi_1}{d_3})(b_{2,\inf}-l\frac{\chi_2}{d_3})-lk\frac{\chi_1\chi_2}{d_3^2}}+\frac{\frac{k\chi_2}{d_3}\left(a_{0,\sup}
-{ (a_{2,\inf}+l\frac{\chi_1}{d_3}) } {\underbar s_2}-k\frac{\chi_1}{d_3}{\underbar s_1}\right)}{(a_{1,\inf}-k\frac{\chi_1}{d_3})(b_{2,\inf}-l\frac{\chi_2}{d_3})-lk\frac{\chi_1\chi_2}{d_3^2}}\cr
\underbar s_1=\frac{\left(a_{0,\inf}-{ (a_{2,\sup}+l\frac{\chi_1}{d_3})}{\bar s_2}-k\frac{\chi_1}{d_3}{\bar s_1}\right)(b_{2,\sup}-l\frac{\chi_2}{d_3})}{(a_{1,\sup}-k\frac{\chi_1}{d_3})(b_{2,\sup}-l\frac{\chi_2}{d_3})-lk\frac{\chi_1\chi_2}
{d_3^2}}+\frac{\frac{l\chi_1}{d_3}\left(b_{0,\inf}-{ (b_{1,\sup}+k\frac{\chi_2}{d_3}) }{\bar s_1}-l\frac{\chi_2}{d_3}{\bar s_2}\right)}{(a_{1,\sup}-k\frac{\chi_1}{d_3})(b_{2,\sup}-l\frac{\chi_2}{d_3})-lk\frac{\chi_1\chi_2}{d_3^2}}\cr
\underbar s_2=\frac{\left(b_{0,\inf}-{ (b_{1,\sup}+k\frac{\chi_2}{d_3})}{\bar s_1}-l\frac{\chi_2}{d_3}{\bar s_2}\right)(a_{1,\sup}-k\frac{\chi_1}{d_3})} {(a_{1,\sup}-k\frac{\chi_1}{d_3})(b_{2,\sup}-l\frac{\chi_2}{d_3})-lk\frac{\chi_1\chi_2}{d_3^2}}+\frac{\frac{k\chi_2}{d_3}\left(a_{0,\inf}
-{ (a_{2,\sup}+l\frac{\chi_1}{d_3}) } {\bar s_2}-k\frac{\chi_1}{d_3}{\bar s_1}\right)}{(a_{1,\sup}-k\frac{\chi_1}{d_3})(b_{2,\sup}-l\frac{\chi_2}{d_3})-lk\frac{\chi_1\chi_2}{d_3^2}}.
\end{cases}
\end{equation}
  Then $0<\underbar s_1\le \bar s_1$, $0<\underbar s_2  \le \bar s_2$, and
   for any $\epsilon>0$, $t_0\in\mathbb{R}$, and  $u_0,v_0\in C^0(\bar \Omega)$ with $\inf u_0>0$, $\inf v_0> 0,$   there exists $ t_{\epsilon,\overline u_0,\overline v_0,\underline u_0,\underline v_0}$,   such that \eqref{attracting-set-eq00}-\eqref{invariant-set-eq01}
   hold with $\bar r_1,\bar r_2,\underbar r_1$, and $\underbar r_2$ being replaced by  $\bar s_1,\bar s_2,\underbar s_1$, and $\underbar s_2$,
   respectively.
 \end{itemize}
 \end{theorem}

\begin{remark}
\label{attracting-rectangle-rk}
\begin{itemize}
\item[(1)]   Under the assumptions in Theorem \ref{attracting-rectangle-thm}(1),
$(\bar r_1,\bar r_2)$ is the unique positive equilibrium of the  system,
\begin{equation*}
\begin{cases}
u_t=u\Big( a_{0,\sup}-\big(a_{1,\inf}-k\frac{\chi_1 }{d_3}\big)u -a_{2,\inf}{\underbar r_2}-k\frac{\chi_1 }{d_3}{\underbar r_1}\Big)\\
v_t=v\Big( b_{0,\sup}-\big(b_{2,\inf}-l\frac{\chi_2 }{d_3}\big)v -b_{1,\inf}{\underbar r_1}-l\frac{\chi_2 }{d_3}{\underbar r_2}\Big),
\end{cases}
\end{equation*}
hence,
$$
\bar r_1<\bar A_1,\quad \bar r_2<\bar A_2,
$$
and
$(\underbar r_1,\underbar r_2)$  is the unique positive equilibrium of the system,
\begin{equation*}
\begin{cases}
u_t=u\Big( a_{0,\inf}-\big(a_{1,\sup}-k\frac{\chi_1 }{d_3}\big)u -a_{2,\sup}{\bar r_2}-k\frac{\chi_1 }{d_3}{\bar r_1}\Big)\\
v_t=v\Big( b_{0,\inf}-\big(b_{2,\sup}-l\frac{\chi_2 }{d_3}\big)v -b_{1,\sup}{\bar r_1}-l\frac{\chi_2 }{d_3}{\bar r_2}\Big).
\end{cases}
\end{equation*}

\item[(2)]  Under the assumptions in Theorem \ref{attracting-rectangle-thm}(2),
$(\bar s_1,\bar s_2)$  is the unique positive equilibrium of the system,
\begin{equation*}
\begin{cases}
u_t=u\Big(a_{0,\sup}-{(a_{2,\inf}+l\frac{\chi_1}{d_3})} {\underbar s_2}-k\frac{\chi_1}{d_3}{\underbar s_1}-(a_{1,\inf}-k\frac{\chi_1}{d_3}) u+l\frac{\chi_1}{d_3}v\Big) \\
v_t= v\Big(b_{0,\sup}-{ (b_{1,\inf}+k\frac{\chi_2}{d_3})} {\underbar s_1}-l\frac{\chi_2}{d_3}{\underbar s_2}-(b_{2,\inf}-l\frac{\chi_2}{d_3})v
+{ k \frac{\chi_2}{d_3}}u\Big),
\end{cases}
\end{equation*}
hence,
$$
\bar s_1<\bar B_1,\quad \bar s_2<\bar B_2,
$$
and
$(\underbar s_1,\underbar s_2)$  is the unique positive equilibrium of the system,
\begin{equation*}
\begin{cases}
u_t=u\Big(a_{0,\inf}-{ (a_{2,\sup}+l\frac{\chi_1}{d_3})} {\bar s_2}-k\frac{\chi_1}{d_3}{\bar s_1}-(a_{1,\sup}-k\frac{\chi_1}{d_3}) u+l\frac{\chi_1}{d_3}v\Big) \\
v_t= v\Big(b_{0,\inf}-{ (b_{1,\sup}+k\frac{\chi_2}{d_3})} {\bar s_1}-l\frac{\chi_2}{d_3}{\bar s_2}-(b_{2,\sup}-l\frac{\chi_2}{d_3})v+{ k \frac{\chi_2}{d_3}}u\Big).
\end{cases}
\end{equation*}

\item[(3)]   When $\chi_1=\chi_2=0$,
    $$
     \underbar r_1=\underbar s_1=\frac{a_{0,\inf}b_{2,\inf}-a_{2,\sup}b_{0,\sup}}{a_{1,\sup}b_{2,\inf}-a_{2,\sup}b_{1,\inf}},\quad
    \bar r_1=\bar s_1=\frac{a_{0,\sup}b_{2,\sup}-a_{2,\inf}b_{0,\inf}}{a_{1,\inf}b_{2,\sup}-a_{2,\inf}b_{1,\sup}},
    $$
    $$
    \underbar r_2=\underbar s_2=\frac{a_{1,\inf}b_{0,\inf}-a_{0,\sup}b_{1,\sup}}{a_{1,\inf}b_{2,\sup}-a_{2,\inf}b_{1,\sup}},\quad
    \bar r_2=\bar s_2=\frac{a_{1,\sup}b_{0,\sup}-a_{0,\inf}b_{1,\inf}}{a_{1,\sup}b_{2,\inf}-a_{2,\sup}b_{1,\inf}}.
    $$
    Thus Theorem \ref{attracting-rectangle-thm} recovers the result on ultimate bounds of solutions of \eqref{u-v-eq00} in \cite{Ahm}. Note that this result can be proven directly by using the competitive comparison principle. {Note also that, in this case, $(\bar r_1,\underbar r_2)$ is the unique coexistence state of
    $$
    \begin{cases}
    u_t=u(a_{0,\sup}-a_{1,\inf}u-a_{2,\inf}v)\cr
    v_t=v(b_{0,\inf}-b_{1,\sup}u-b_{2,\sup}v)
    \end{cases}
    $$
    and $(\underbar r_1,\bar r_2)$ is the unique coexistence state of
    $$
    \begin{cases}
    u_t=u(a_{0,\inf}-a_{1,\sup}u-a_{2,\sup}v)\cr
    v_t=v(b_{0,\sup}-b_{1,\inf}u-b_{2,\inf}v).
    \end{cases}
    $$
    }

    \item[(4)] When the coefficients are constants, i.e $a_i(t,x)=a_i$ and $b_i(t,x)=b_i$ $(i=0,1,2)$, we  have
    $$\underbar r_1=\bar r_1=\underbar s_1=\bar s_1=\frac{a_0b_2-a_2b_0}{b_2a_1-b_1a_2},$$
    and
      $$\underbar r_2=\bar r_2=\underbar s_2=\bar s_2=\frac{b_0a_1-b_1a_0}{b_2a_1-b_1a_2}.
      $$
    Thus Theorem \ref{attracting-rectangle-thm} implies the uniqueness and stability of coexistence states.
     { In this case, by  Remark \ref{attracting-region-rk}(2),  \eqref{attracting-rectabgle-eq1} has a unique solution iff
    \begin{equation}\label{optimal-rect-cond-00c}
    \begin{cases}
      \frac{a_2}{b_2}<\frac{a_0}{b_0}<\frac{a_1}{b_1},\cr
      \big(a_1-2k\frac{\chi_1}{d_3}\big)\big(b_2-2l\frac{\chi_2}{d_3}\big)>a_2b_1.
      \end{cases}
    \end{equation}
    Note that \eqref{optimal-rect-cond-00c} are the sufficient conditions for
    the asymptotic stability and uniqueness of the constant positive steady states in \cite[Theorem 1.3]{ITBRS17} and \cite{TBJLMM16}.}
\end{itemize}
\end{remark}

The second main theorem of the current paper is on the uniqueness and stability of coexistence states of \eqref{u-v-w-eq00}.

\begin{theorem} [Stability and uniqueness of coexistence states]
\label{thm-nonlinear-stability-001} $\,$
\begin{itemize}
\item[(1)]
Assume (H5).
 Furthermore, assume that
\begin{equation}
\label{stability-cond-2-eq1}
\limsup_{t-s\to\infty}\frac{1}{t-s}\int_{s}^{t}\max\{Q_1(\tau)-q_1(\tau),Q_2(\tau)-q_2(\tau)\}d\tau <0,
\end{equation}
where
 \begin{equation}
 \label{q-eq1}
 q_1(t)=2a_{1,\inf}(t)\underbar r_1+a_{2,\inf}(t)\underbar r_2+\frac{\chi_1\left(k\underbar r_1+l\underbar r_2\right)}{2d_3},
 \end{equation}
 \begin{align}
 \label{q-eq2}
 Q_1(t)=a_{0,\sup}(t)+\frac{\chi_1}{2d_3}\big(k{\bar r_1}+l{\bar r_2}\big)+\frac{k^2}{4\lambda d_3}\Big(\frac{\chi_1^2{\bar r_1}^2}{d_1 }+\frac{\chi_2^2{\bar r_2}^2}{d_2}\Big)+\frac{a_{2,\sup}(t){\bar r_1}+b_{1,\sup}(t){\bar r_2}}{2},
\end{align}
 \begin{equation}
 \label{q-eq3}
 q_2(t)=2b_{2,\inf}(t)\underbar r_2+b_{1,\inf}(t)\underbar r_1+\frac{\chi_2\left(k\underbar r_1+l\underbar r_2\right)}{2d_3},
 \end{equation}
 and
 \begin{align}
 \label{q-eq4}
 Q_2(t)=b_{0,\sup}(t)+\frac{\chi_2}{2d_3}\big(k{\bar r_1}+l{\bar r_2}\big)+\frac{l^2}{4\lambda d_3}\Big(\frac{\chi_1^2{\bar r_1}^2}{d_1 }+\frac{\chi_2^2{\bar r_2}^2}{d_2}\Big)+\frac{a_{2,\sup}(t){\bar r_1}+b_{1,\sup}(t){\bar r_2}}{2}.
\end{align}
Then \eqref{u-v-w-eq00} has a unique coexistence state $(u^{**}(x,t),v^{**}(x,t), w^{**}(x,t))$, and,  for any
$t_0\in\RR$ and  $u_0, v_0 \in { C^+(\bar{\Omega})}$ with
$u_0,v_0\not\equiv 0$,  the  global classical solution
$(u(x,t;t_0,u_0,v_0)$,  $v(x,t;t_0,u_0,v_0)$, $w(x,t;t_0,u_0,v_0))$ of $\eqref{u-v-w-eq00}$ satisfies
  \begin{equation}
  \label{new-global-stability-2-eq1}
  \lim_{t \to \infty}\Big(\|u(\cdot,t;t_0,u_0,v_0)-u^{**}(\cdot,t)\|_{C^0(\bar\Omega)}+\|v(\cdot,t;t_0,u_0,v_0)-v^{**}(\cdot,t)\|_{C^0(\bar\Omega)}\Big)=0,
  \end{equation}
and
\begin{equation}
  \label{new-global-stability-2-eq2}
  \lim_{t \to \infty}\|w(\cdot,t;t_0,u_0,v_0)-w^{**}(\cdot,t)\|_{C^0(\bar\Omega)}=0.
  \end{equation}

  \item[(2)] Assume (H6).
Furthermore, assume that \eqref{stability-cond-2-eq1}-\eqref{q-eq4} hold with $\bar r_1$, $\bar r_2$, $\underbar r_1$, and $\underbar r_2$ being replaced by
$\bar s_1$, $\bar s_2$, $\underbar s_1$, and $\underbar s_2$, respectively, where
$\underbar s_i$ and $\bar s_i$ $(i=1,2)$ are as in Theorem \ref{attracting-rectangle-thm}(2).
 Then the conclusion in (1) also holds.

 \item[(3)] Assume (H7). Then
\eqref{u-v-w-eq00} has a  unique spatially homogeneous coexistence state  $(u^{**}(t),v^{**}(t)$, $w^{**}(t))$, and  for any
$t_0\in\RR$ and  $u_0,v_0 \in {C^+(\bar{\Omega})}$ with
$u_0,v_0\not\equiv 0$,  the  unique global classical solution
$(u(x,t;t_0,u_0,v_0), v(x,t;t_0,u_0,v_0),w(x,t;t_0,u_0,v_0))$ of $\eqref{u-v-w-eq00}$ satisfies
\begin{equation}
\label{global-stability-1-eq1}
\lim_{t \to \infty}\big(  \left\| u(\cdot,t;t_0,u_0,v_0)-u^{**}(t) \right\|_{C^0(\bar\Omega)} +\left\| v(\cdot,t;t_0,u_0,v_0)-v^{**}(t) \right\|_{C^0(\bar\Omega)}\big)=0,
\end{equation}
\begin{equation}
\label{global-stability-1-eq2}
\lim_{t \to \infty}\left\|{ w}(\cdot,t;t_0,u_0,v_0)-ku^{**}(t)-lv^{**}(t) \right\|_{C^0(\bar\Omega)}=0.
\end{equation}
\end{itemize}
\end{theorem}

\begin{remark}
\label{rmk-stability}
\begin{itemize}
\item[(1)]{ Assume (H7). \eqref{stability-cond-1-eq1} implies that
\begin{equation*}
\begin{cases}
u_t=u( a_0(t)-a_1(t)u -a_2(t)v)\\
v_t=v( b_0(t)-b_1(t)u -b_2(t)v)\\
\end{cases}
\end{equation*}
has a positive entire solution $(u^{**}(t),v^{**}(t))$ which is  globally stable (see Lemma \ref{persistence-lm6}). Thus $(u^{**}(t),v^{**}(t),w^{**}(t))$ with $w^{**}(t)=\frac{ku^{**}(t)+lv^{**}(t)}{\lambda},$ is a positive entire solution of \eqref{u-v-w-eq00} in the case of space homogeneous coefficients,  i.e,  $a_i(t,x)=a_i(t)$ and $b_i(t,x)=b_i(t)$. The uniqueness results is new even for the case $\chi_1=\chi_2=0$ with general time dependence. { When the coefficients are periodic, Alvarez and Lazer proved in \cite{AlLA86} the uniqueness of the entire solution $(u^{**}(t),v^{**}(t))$ only under the assumption \eqref{stability-cond-1-eq1}. It { remains open}   whether such uniqueness result holds even in the case of $\chi_1=\chi_2=0$ with general time dependence under only the assumption \eqref{stability-cond-1-eq1}. { The  arguments in the proof of Theorem \ref{thm-nonlinear-stability-001}(3) are similar to those in \cite{TW12}. } }}

\item[(2)] { The arguments in \cite{TW12} as well as the arguments in  \cite{TBJLMM16} and \cite{ITBRS17} are difficult to be applied in the general nonhomogeneous case. We utilized a new approach to prove
 the stability and uniqueness of coexistence states in this later case, namely,   we first obtain optimal attracting rectangle by iterating the
so called eventual comparison method (see the proof of Theorem \ref{attracting-rectangle-thm}), next we  prove  the stability of coexistence states
  in $L^2$-norm by applying Green's Theorem and Young's inequality and using the optimal attracting rectangle established in Theorem \ref{attracting-rectangle-thm}, and then we prove the stability and uniqueness of coexistence states in $L^\infty$-norm.
  By this new approach, we also obtain some new result about the uniqueness and stability of coexistence states of \eqref{u-v-eq00}
  (see Corollary \ref{stability-cor}).}

\item[(3)] \eqref{stability-cond-1-eq2} implies (H2). It is the analogue of the condition $a_{1,\inf}>\frac{2\chi_1 k}{d_3}$ for the global stability of
 the unique spatially homogeneous entire positive solution of the following one species chemotaxis model,
 \begin{equation*}
 \begin{cases}
 u_t=d_1\Delta u-\chi_1\nabla\cdot (u \nabla w)+u\Big(a_0(t)-a_1(t)u\Big),\quad x\in \Omega\cr
0=d_3\Delta w+k u-\lambda w,\quad x\in \Omega \cr
\frac{\p u}{\p n}=\frac{\p w}{\p n}=0,\quad x\in\p\Omega,
 \end{cases}
 \end{equation*}
 (see \cite[Theorem 1.4]{ITBWS16}).

\item[(4)] When $\chi_1=\chi_2=0$,
 \eqref{stability-cond-2-eq1} becomes
 \begin{equation}
\label{new-condition-eq1}
\begin{cases}
\overline{\lim}_{t-s\to\infty}\frac{1}{t-s}\int_{s}^{t}\Big\{a_{0,\sup}(\tau)+\frac{a_{2,\sup}(\tau)}{2}\bar r_1-2a_{1,\inf}(\tau)\underbar r_1+\frac{b_{1,\sup}(\tau)}{2}\bar r_2-a_{2,\inf}(\tau)\underbar r_2\Big\}d\tau <0\cr
  \overline{\lim}_{t-s\to\infty}\frac{1}{t-s}\int_{s}^{t}\Big\{b_{0,\sup}(\tau)+\frac{b_{1,\sup}(\tau)}{2}\bar r_2-2b_{2,\inf}(\tau)\underbar r_2+\frac{a_{2,\sup}(\tau)}{2}\bar r_1-b_{1,\inf}(\tau)\underbar r_1\Big\}d\tau <0.
\end{cases}
\end{equation}
If furthermore the coefficients are time homogeneous i.e $a_i(t,x)=a_i(x)$ and $b_i(t,x)=b_i(x),$  then \eqref{stability-cond-2-eq1} becomes
\begin{equation}
\label{new-condition-eq2}
\begin{cases}
a_{0,\sup}+\frac{a_{2,\sup}}{2}\bar r_1+\frac{b_{1,\sup}}{2}\bar r_2 <2a_{1,\inf}\underbar r_1+a_{2,\inf}\underbar r_2 \cr
b_{0,\sup}+\frac{b_{1,\sup}}{2}\bar r_2+\frac{a_{2,\sup}}{2}\bar r_1 <2b_{2,\inf}\underbar r_2+b_{1,\inf}\underbar r_1.
\end{cases}
\end{equation}
\end{itemize}
\end{remark}

We have the following corollary for the uniqueness and stability of coexistence states of \eqref{u-v-eq00}, which is new in the
general space dependence case.

\begin{corollary}
\label{stability-cor} Consider \eqref{u-v-eq00}.
Assume that $\frac{a_{0,\sup}}{a_{0,\inf}}<2\frac{a_{1,\inf}}{a_{1,\sup}}$ and
$\frac{b_{0,\sup}}{b_{0,\inf}}<2\frac{b_{2,\inf}}{b_{2,\sup}}$.
Then \eqref{u-v-eq00} has a unique stable
coexistence state provided that the competition coefficients $a_2$ and $b_1$ are such small so that  \eqref{stability-cond-1-eq1} and the following hold,
\begin{equation*}
\begin{cases}
  a_{2,\sup}\big(\frac{\bar r_1}{2}+\frac{2a_{1,\inf}b_{0,\sup}-a_{0,\sup}b_{1,\inf}}{a_{1,\sup}b_{2,\inf}-a_{2,\sup}b_{1,\inf}}
\big)+\frac{b_{1,\sup}}{2}\bar r_2-a_{2,\inf}\underbar r_2 <b_{2,\inf}\frac{2a_{1,\inf}a_{0,\inf}-a_{0,\sup}a_{1,\sup}}{a_{1,\sup}b_{2,\inf}-a_{2,\sup}b_{1,\inf}} \cr\cr
  b_{1,\sup}\big(\frac{\bar r_2}{2}+\frac{2b_{2,\inf}a_{0,\sup}-b_{0,\sup}a_{2,\inf}}{b_{2,\sup}a_{1,\inf}-b_{1,\sup}a_{2,\inf}}
\big)+\frac{a_{2,\sup}}{2}\bar r_1-b_{1,\inf}\underbar r_1 <a_{1,\inf}\frac{2b_{2,\inf}b_{0,\inf}-b_{0,\sup}b_{2,\sup}}{b_{2,\sup}a_{1,\inf}-b_{1,\sup}a_{2,\inf}}.
\end{cases}
\end{equation*}
\end{corollary}

The rest of the paper is organized as follows. In section 2, we recall some important results in \cite{ITBWS17a} to be used to prove our main theorems. We discuss the optimal attracting rectangles  in section 3. It is here that we prove Theorem \ref{attracting-rectangle-thm}. In section 4, we prove  our main  Theorem  \ref{thm-nonlinear-stability-001} about stability and uniqueness of coexistence in the general inhomogeneous case. We also prove Corollary \ref{stability-cor} in this section.
Finally, we discuss the conditions for   \eqref{attracting-rectabgle-eq1} to have  a unique solution in section 5. Note that such conditions are generic (see Proposition \ref{r-eq-solu}).

\medskip

{\section{Preliminary}
 Consider the following system of ODEs induced from system \eqref{u-v-w-eq00},
\begin{equation}
\label{ode00}
\begin{cases}
\overline{u}'=\frac{\chi_1}{d_3} \overline{u}\big(k \overline {u}+l\overline v-k\underline{u}-l\underline{v}\big)+ \overline{u}\big[a_{0,\sup}(t)-a_{1,\inf}(t)\overline u
-a_{2,\inf}(t)\underline{v}\big]\\
\underline{u}'=\frac{\chi_1}{d_3} \underline{u}\big(k \underline {u}+l\underline v-k\overline{u}-l\overline{v}\big)+ \underline{u}\big[a_{0,\inf}(t)-a_{1,\sup}(t)\underline u
-a_{2,\sup}(t)\overline{v}\big]\\
\overline{v}'=\frac{\chi_2}{d_3} \overline{v}\big(k \overline {u}+l\overline v-k\underline{u}-l\underline{v}\big)+ \overline{v}\big[b_{0,\sup}(t) -b_{1,\inf}(t)\underline{u}-b_{2,\inf}(t)\overline v\big]\\
\underline{v}'=\frac{\chi_2}{d_3} \underline{v}\big(k \underline {u}+l\underline v-k\overline{u}-l\overline{v}\big)+ \underline{v}\big[b_{0,\inf}(t)-b_{1,\sup}(t)\overline{u}-b_{2,\sup}(t)\underline v\big].
\end{cases}
\end{equation}
For convenience, we let
\begin{align*}
&\left(\overline{u}(t),\underline{u}(t),\overline{v}(t),\underline{v}(t)\right)\\
&=\left(\overline{u}\left(t;t_0,\overline{u}_0,\underline{u}_0,\overline{v}_0,\underline{v}_0\right),\underline{u}\left(t;t_0,\overline{u}_0,\underline{u}_0,\overline{v}_0,\underline{v}_0\right),\overline{v}\left(t;t_0,\overline{u}_0,\underline{u}_0,\overline{v}_0,\underline{v}_0\right),\underline{v}\left(t;t_0,\overline{u}_0,\underline{u}_0,\overline{v}_0,
\underline{v}_0\right)\right)
\end{align*}
be the solution of \eqref{ode00}
with initial condition
\begin{align}\label{initial-ode00}
&\left(\overline{u}\left(t_0;t_0,\overline{u}_0,\underline{u}_0,\overline{v}_0,\underline{v}_0\right),\underline{u}\left(t_0;t_0,\overline{u}_0,
\underline{u}_0,\overline{v}_0,\underline{v}_0\right),\overline{v}\left(t_0;t_0,\overline{u}_0,\underline{u}_0,\overline{v}_0,\underline{v}_0\right),
\underline{v}\left(t_0;t_0,\overline{u}_0,\underline{u}_0,\overline{v}_0,
\underline{v}_0\right)\right)\nonumber\\
&=\left(\overline{u}_0,\underline{u}_0,\overline{v}_0,\underline{v}_0\right) \in \mathbb{R}^4_+.
\end{align}
Then for given $t_0\in\mathbb{R}$ and  $\left(\overline{u}_0,\underline{u}_0,\overline{v}_0,\underline{v}_0\right) \in \mathbb{R}^4_+,$ there exists $T_{\max}\left(t_0,\overline{u}_0,\underline{u}_0,\overline{v}_0,\underline{v}_0\right)>0$ such that \eqref{ode00} has a unique classical solution
$\left(\overline{u}(t),\underline{u}(t),\overline{v}(t),\underline{v}(t)\right)$ on $(t_0,t_0+T_{\max}\left(t_0,\overline{u}_0,\underline{u}_0,\overline{v}_0,\underline{v}_0\right))$ satisfying \eqref{initial-ode00}. Moreover if
$T_{\max}\left(t_0,\overline{u}_0,\underline{u}_0,\overline{v}_0,\underline{v}_0\right)<\infty,$ then
\begin{equation}\label{blow-creterion-ode00}
\limsup_{t \nearrow T_{\max}\left(t_0,\overline{u}_0,\underline{u}_0,\overline{v}_0,\underline{v}_0\right)}
\left(|\overline{u}(t_0+t)|+|\underline{u}(t_0+t)|+|\overline{v}(t_0+t)|+|\underline{v}(t_0+t)|\right)=\infty.
\end{equation}

Then we have the following important lemma from \cite{ITBWS17a}.

\begin{lemma}\cite[{ Lemma 2.2}]{ITBWS17a}
\label{lem-1-ode00} Let $\left(\overline{u}(t),\underline{u}(t),\overline{v}(t),\underline{v}(t)\right)$ be the solution of \eqref{ode00} which satisfies \eqref{initial-ode00}.
\begin{itemize}
\item[(i)] If $0\leq\underline{u}_0 \leq \overline{u}_0\quad \text{and} \quad 0\leq \underline{v}_0 \leq \overline{v}_0$, then
 $ 0\leq\underline{u}(t )\leq \overline{u}(t) \quad \text{and} \quad 0\leq \underline{v}(t)\leq \overline{v}(t)$ for all $t \in [t_0,t_0+T_{\max}\left(\overline{u}_0,\underline{u}_0,\overline{v}_0,\underline{v}_0\right)).$

\item[(ii)] If { (H2)} holds, then $T_{\max}\left(t_0,\overline{u}_0,\underline{u}_0,\overline{v}_0,\underline{v}_0\right)=\infty$ and
 $$
\limsup_{t\to\infty} \overline{u}(t) \leq \bar B_1,\quad \limsup_{t\to\infty}\overline{v}(t)\leq  \bar B_2,
$$
 where $\bar B_1$ and $\bar B_2$ are as in \eqref{A1-overbar-0} and \eqref{A2-overbar-0}, respectively.
\end{itemize}
\end{lemma}

{
\begin{lemma}\cite[Proof Theorem 1.1(1)]{ITBRS17}
\label{lem-1-ode000}
Assume { (H2)}. Given $t_0 \in \mathbb{R},$ $u_0,v_0 \in {C^+(\bar{\Omega})},$ let $\overline{u}_0=\max_{x \in \bar \Omega}u_0(x),$ $\underline{u}_0=\min_{x \in \bar \Omega}u_0(x),$ $\overline{v}_0=\max_{x \in \bar \Omega}v_0(x)$ , $\underline{v}_0=\min_{x \in \bar \Omega}v_0(x)$ and let $\left(\overline{u}(t),\underline{u}(t),\overline{v}(t),\underline{v}(t) \right)$ be solution of \eqref{ode00} satisfying initial condition \eqref{initial-ode00}. Then if $(u(x,t),v(x,t)$, $w(x,t))$ is the solution of equation \eqref{u-v-w-eq00} with initials $u(\cdot,t_0)=u_0$ and $v(\cdot,t_0)=v_0,$ we have
\[
0\leq \underline{u}(t) \leq u(x,t) \leq \overline{u}(t) \quad \text{and} \quad  0 \leq \underline{v}(t) \leq v(x,t) \leq \overline{v}(t) \, , \forall  x \in \bar \Omega \, \,\,  t \geq t_0.
\]
\end{lemma}
\begin{proof}
By the similar arguments as those in \cite[Theorem 1.1(1)]{ITBRS17}, under the condition { (H2)} we have
\[
0\leq \underline{u}(t) \leq u(x,t) \leq \overline{u}(t) \quad \text{and} \quad  0 \leq \underline{v}(t) \leq v(x,t) \leq \overline{v}(t) \, , \forall  x \in \bar \Omega \, \,\,  t \in (t_0, t_0+T_{\max}).
\]
By { (H2)} and  Lemma \ref{lem-1-ode00}, we get $T_{\max}=\infty.$
\end{proof}
}
Next, we have the following lemma about existence and stability of coexistence states when the coefficients are space independent.
\begin{lemma}\cite[Lemma 4.1 ]{ITBWS17a}
\label{persistence-lm6}
Consider
\begin{equation}
\begin{cases}
\label{u-v-ode}
u_t=u\big(a_0(t)-a_1(t)u-a_2(t)v\big)\cr
v_t=v\big(b_0(t)-b_1(t)u-b_2(t)v\big).
\end{cases}
 \end{equation}
Assume \eqref{stability-cond-1-eq1} is satisfied.
 Then there is a  { strictly} positive entire solution $(u^{**}(t),v^{**}(t))$ of \eqref{u-v-ode}. Moreover, for any $u_0,v_0>0$ and $t_0\in\RR$,
 $$
 (u(t;t_0,u_0,v_0),v(t;t_0,u_0,v_0))-(u^{**}(t),v^{**}(t))\to 0
 $$
 as $t\to\infty$, where $(u(t;t_0,u_0,v_0)$, $v(t;t_0,u_0,v_0))$ is the solution of \eqref{u-v-ode} with
 $(u(t_0;t_0,u_0,v_0)$, $v(t_0;t_0,u_0,v_0))=(u_0,v_0)$.
 In addition, if $a_i(t)$ and $b_i(t)$ are almost periodic, then so is $(u^{**}(t)$, $v^{**}(t))$.
\end{lemma}

\section{Optimal attracting rectangle and proof of Theorem \ref{attracting-rectangle-thm}}

In this section, we construct optimal attracting rectangles for \eqref{u-v-w-eq00} and  prove  Theorem \ref{attracting-rectangle-thm}.  We first prove two important lemmas.

\begin{lemma}
\label{iteration-lm-00}
Consider \eqref{u-v-w-eq00}. For given $u_0,v_0 \in { C^+(\bar{\Omega})},$  let $\overline{u}_0=\max_{x \in \bar \Omega}u_0(x),$ $\underline{u}_0=\min_{x \in \bar \Omega}u_0(x),$ $\overline{v}_0=\max_{x \in \bar \Omega}v_0(x)$ , $\underline{v}_0=\min_{x \in \bar \Omega}v_0(x)$.

\begin{itemize}
\item[(1)] Assume (H5).  Let $\underbar r_1^0=\underbar r_2^0=0$, $\bar r_1^0=\bar A_1$, $\bar r_2^0=\bar A_2$,  and
\begin{equation}
\label{r-iteration-eq}
\begin{cases}
  \bar r_1^n=\frac{a_{0,\sup}-a_{2,\inf}\underbar r_2^{n-1}-k\frac{\chi_1}{d_3}\underbar r_1^{n-1}}{a_{1,\inf}-k\frac{\chi_1}{d_3}}\cr
 \bar r_2^n=\frac{b_{0,\sup}-b_{1,\inf}\underbar r_1^{n-1}-k\frac{\chi_1}{d_3}\underbar r_2^{n-1}}{b_{2,\inf}-l\frac{\chi_2}{d_3}}\cr
\underbar r_1^n=\frac{a_{0,\inf}-a_{2,\sup}\bar r_2^{n}-k\frac{\chi_1}{d_3}\bar r_1^{n}}{a_{1,\sup}-k\frac{\chi_1}{d_3}}\cr
\underbar r_2^n=\frac{b_{0,\inf}-b_{1,\sup}\bar r_1^{n}-l\frac{\chi_2}{d_3}\bar r^{n}_{ 2}}{b_{2,\sup}-l\frac{\chi_2}{d_3}}
\end{cases}
\end{equation}
for $n=1,2,\cdots$.
Then
\begin{equation}
\label{itern-eq05aa}
\begin{cases}
0<\underbar r_1^{n-1}\leq \underbar r_1^n\leq \bar r_1^n\leq \bar r^{n-1}_1\le \bar A_1\cr
0<\underbar r_2^{n-1}\leq \underbar r_2^n\leq \bar r_2^n\leq \bar r_2^{n-1}\le \bar A_2
\end{cases}
\end{equation}
for $n={ 2,\cdots}$,
and for any given $u_0,v_0 \in C^+(\bar{\Omega})$ with $\inf u_0>0$, $\inf v_0>0$,   $\epsilon>0$,  and $n\in \mathbb{N}$ { with $n\geq 1$,}  there exists
$ t^n_{\epsilon,\overline u_0,\overline v_0,\underline u_0,\underline v_0}\geq t^{n-1}_{\epsilon,\overline u_0,\overline v_0,\underline u_0,\underline v_0}$ ($t^0_{\epsilon,\overline u_0,\overline v_0,\underline u_0,\underline v_0}=0$)
 such that
\begin{equation}\label{itern-eq00a}
\begin{cases}
\underbar r_1^n -\epsilon \le u(x,t;t_0,u_0,v_0) \le \bar r_1^n+\epsilon\cr
\underbar r_2^n -\epsilon \le v(x,t;t_0,u_0,v_0) \le \bar r_2^n+\epsilon,
\end{cases}
\end{equation}
for all $x\in\bar\Omega$, $t_0 \in \mathbb{R}$ and $t\ge t_0+t^n_{\epsilon,\overline u_0,\overline v_0,\underline u_0,\underline v_0}$.

\item[(2)] Assume (H6).  Let $\underbar s_1^0=\underbar s_2^0=0$,  $\bar s_1^0=\bar B_1$, $\bar s_2^0=\bar B_2$, and
 \begin{equation}
 \label{s-iteration-eq}
 \begin{cases}
 \bar s_1^n=\frac{\left(a_{0,\sup}-{ (a_{2,\inf}+l\frac{\chi_1}{d_3})}{\underbar s_2^{n-1}}-k\frac{\chi_1}{d_3}
 {\underbar s_1^{n-1}}\right)(b_{2,\inf}-l\frac{\chi_2}{d_3})}{(a_{1,\inf}-k\frac{\chi_1}{d_3})(b_{2,\inf}-l\frac{\chi_2}{d_3})-lk\frac{\chi_1\chi_2}
{d_3^2}}+\frac{\frac{l\chi_1}{d_3}\left(b_{0,\sup}-{ (b_{1,\inf}+k\frac{\chi_2}{d_3}) }{\underbar s_1^{n-1}}-l\frac{\chi_2}{d_3}{\underbar s_2^{n-1}}\right)}{(a_{1,\inf}-k\frac{\chi_1}{d_3})(b_{2,\inf}-l\frac{\chi_2}{d_3})-lk\frac{\chi_1\chi_2}{d_3^2}}\cr
\bar s_2^n=\frac{\left(b_{0,\sup}-{ (b_{1,\inf}+k\frac{\chi_2}{d_3})}{\underbar s_1^{n-1}}-l\frac{\chi_2}{d_3}{\underbar s_2^{n-1}}\right)(a_{1,\inf}-k\frac{\chi_1}{d_3})} {(a_{1,\inf}-k\frac{\chi_1}{d_3})(b_{2,\inf}-l\frac{\chi_2}{d_3})-lk\frac{\chi_1\chi_2}{d_3^2}}+\frac{\frac{k\chi_2}{d_3}\left(a_{0,\sup}
-{ (a_{2,\inf}+l\frac{\chi_1}{d_3}) } {\underbar s_2^{n-1}}-k\frac{\chi_1}{d_3}{\underbar s_1^{n-1}}\right)}{(a_{1,\inf}-k\frac{\chi_1}{d_3})(b_{2,\inf}-l\frac{\chi_2}{d_3})-lk\frac{\chi_1\chi_2}{d_3^2}}\cr
\underbar s_1^n=\frac{\left(a_{0,\inf}-{ (a_{2,\sup}+l\frac{\chi_1}{d_3})}{\bar s_2^n}-k\frac{\chi_1}{d_3}{\bar s_1^n}\right)(b_{2,\sup}-l\frac{\chi_2}{d_3})}{(a_{1,\sup}-k\frac{\chi_1}{d_3})(b_{2,\sup}-l\frac{\chi_2}{d_3})-lk\frac{\chi_1\chi_2}
{d_3^2}}+\frac{\frac{l\chi_1}{d_3}\left(b_{0,\inf}-{ (b_{1,\sup}+k\frac{\chi_2}{d_3}) }{\bar s_1^n}-l\frac{\chi_2}{d_3}{\bar s_2^n}\right)}{(a_{1,\sup}-k\frac{\chi_1}{d_3})(b_{2,\sup}-l\frac{\chi_2}{d_3})-lk\frac{\chi_1\chi_2}{d_3^2}}\cr
\underbar s_2^n=\frac{\left(b_{0,\inf}-{ (b_{1,\sup}+k\frac{\chi_2}{d_3})}{\bar s_1^n}-l\frac{\chi_2}{d_3}{\bar s_2^n}\right)(a_{1,\sup}-k\frac{\chi_1}{d_3})} {(a_{1,\sup}-k\frac{\chi_1}{d_3})(b_{2,\sup}-l\frac{\chi_2}{d_3})-lk\frac{\chi_1\chi_2}{d_3^2}}+\frac{\frac{k\chi_2}{d_3}\left(a_{0,\inf}
-{ (a_{2,\sup}+l\frac{\chi_1}{d_3}) } {\bar s_2^n}-k\frac{\chi_1}{d_3}{\bar s_1^n}\right)}{(a_{1,\sup}-k\frac{\chi_1}{d_3})(b_{2,\sup}-l\frac{\chi_2}{d_3})-lk\frac{\chi_1\chi_2}{d_3^2}}
\end{cases}
\end{equation}
for $n=1,2,\cdots$.
Then
\begin{equation}
\label{s-itern-eq05aa}
\begin{cases}
0<\underbar s_1^{n-1}\leq \underbar s_1^n\leq \bar s_1^n\leq \bar s^{n-1}_1\le \bar B_1\cr
0<\underbar s_2^{n-1}\leq \underbar s_2^n\leq \bar s_2^n\leq \bar s_2^{n-1}\le B_2
\end{cases}
\end{equation}
for $n={2,\cdots}$,
and for any given $u_0,v_0 \in C^+(\bar{\Omega})$ with $\inf u_0>0$, $\inf v_0>0$,   $\epsilon>0$,  and $n\in \mathbb{N}$ { with $n\geq 1,$} there exists
$ t^n_{\epsilon,\overline u_0,\overline v_0,\underline u_0,\underline v_0}\geq t^{n-1}_{\epsilon,\overline u_0,\overline v_0,\underline u_0,\underline v_0}$ ($t^0_{\epsilon,\overline u_0,\overline v_0,\underline u_0,\underline v_0}=0$)
 such that
\begin{equation}\label{s-itern-eq00a}
\begin{cases}
\underbar s_1^n -\epsilon \le u(x,t;t_0,u_0,v_0) \le \bar s_1^n+\epsilon\cr
\underbar s_2^n -\epsilon \le v(x,t;t_0,u_0,v_0) \le \bar s_2^n+\epsilon,
\end{cases}
\end{equation}
for all $x\in\bar\Omega$, $t_0 \in \mathbb{R}$ and $t\ge t_0+t^n_{\epsilon,\overline u_0,\overline v_0,\underline u_0,\underline v_0}$.

\end{itemize}
\end{lemma}

\begin{proof}
(1) First of all,  note that $\bar r_1^1=\bar r_1^0$ and $\bar r_2^1=\bar r_2^0$, and by (H5),
$0<\underbar r_1^1\le \bar r_1^1$ and $0<\underbar r_2^1\le\bar r_2^1$. \eqref{itern-eq05aa}
then follows from \eqref{r-iteration-eq} directly.

We then prove \eqref{itern-eq00a}. We do so by induction.

First we   claim that there exists { $ t^1_{\epsilon,\overline u_0,\overline v_0,\underline u_0,\underline v_0}\ge 0$}
 such that
\begin{equation}\label{iter1-eq00}
\begin{cases}
\underbar r_1^1 -\epsilon \le u(x,t;t_0,u_0,v_0) \le \bar r_1^1+\epsilon\cr
\underbar r_2^1 -\epsilon \le v(x,t;t_0,u_0,v_0) \le \bar r_2^1+\epsilon
\end{cases}
\end{equation}
for all $x\in\bar\Omega$, $t_0 \in \mathbb{R}$ and { $t\ge t_0+t^1_{\epsilon,\overline u_0,\overline v_0,\underline u_0,\underline v_0}$}.

In fact,  from the first and third equations of \eqref{u-v-w-eq00}, we get
\begin{equation}\label{iter1-eq01}
   u_t \leq d_1\Delta u -\chi_1 \nabla w \cdot \nabla u +u\big(a_{0,\sup}-(a_{1,\inf}-k\frac{\chi_1}{d_3})u\big).
\end{equation}
Let $u(t;t_0,\overline u_0)$ be the solution of
$$u'=u\big(a_{0,\sup}-(a_{1,\inf}-k\frac{\chi_1}{d_3})u\big)$$
 with $u(t_0;t_0,\overline u_0)=\overline u_0.$ Then by solving, we get
\begin{equation}\label{iter1-eq02}
  u(t;t_0,\overline u_0)=\frac{c_0a}{c_0b-e^{-a(t-t_0)}} \quad \forall \,\, t\geq  t_0,
\end{equation}
where $a=a_{0,\sup},$ $b=a_{1,\inf}-k\frac{\chi_1}{d_3},$ and $c_0=\frac{\overline u_0}{b\overline u_0-a}.$ (Actually $u(t;t_0,\overline u_0)>0$ for all $t>t_0-\frac{\ln(c_0b)}{a}$ and blows up in backward time at $t^*=t_0-\frac{\ln(c_0b)}{a}<t_0.$)
Then it follows from \eqref{iter1-eq01}, { \eqref{iter1-eq02}} and parabolic comparison principle that
$$
u(x,t;t_0,u_0,v_0)\leq \frac{c_0a}{c_0b-e^{-a(t-t_0)}} \quad \forall \,\, t\geq   t_0, \quad \forall \,\, t_0 \in \mathbb{R}.
$$
Thus
$$
u(x,t+t_0;t_0,u_0,v_0)\leq u(t+t_0;t_0,\overline u_0)=\frac{c_0a}{c_0b-e^{-at}} \quad \forall\,\,  t\geq  0 ,\quad \forall\,\,  t_0 \in \mathbb{R}.
$$
Therefore there is  $ t^1_{\epsilon,\overline u_0}>0$ such that
\begin{equation}\label{iter1-eq03}
  u(x,t;t_0,u_0,v_0)\leq \bar r_1^1+\epsilon \,\, \forall t\geq t_0+t^1_{\epsilon,\overline u_0},\,\quad \forall\,\,  t_0 \in \mathbb{R}.
\end{equation}

Similarly using the second and third equation of \eqref{u-v-w-eq00},  there exists  $t^1_{\epsilon,\overline v_0}>0$ such that
\begin{equation}\label{iter1-eq04}
  v(x,t;t_0,u_0,v_0)\leq \bar r_2^1+\epsilon \,\, \forall\,\,  t\geq t_0+t^1_{\epsilon,\overline v_0},\,\quad \forall\,\,  t_0 \in \mathbb{R}.
\end{equation}

Choose $0<\tilde \epsilon\le \epsilon$ such that
$$
\frac{a_{0,\inf}-a_{2,\sup}\bar r_2^1-k\frac{\chi_1}{d_3}\bar r_1^1-\tilde \epsilon\big({ a_{2,\sup}+k\frac{\chi_1}{d_3}}\big)}{a_{1,\sup}-\frac{k\chi_1}{d_3}}-\tilde \epsilon\ge \underbar r_1^1-\epsilon.
$$
Let $t^1_{\tilde \epsilon,\overline u_0,\overline v_0}=\max\{t_{\tilde \epsilon,\overline u_0},t_{\tilde \epsilon,\overline v_0}\}.$ Then for $t\geq t^1_{\tilde \epsilon,\overline u_0,\overline v_0}$,   from \eqref{iter1-eq03}, \eqref{iter1-eq04}, the first and third equations of \eqref{u-v-w-eq00}, we get
\begin{equation*}
   u_t \geq d_1\Delta u -\chi_1 \nabla w \cdot \nabla u +u\big(a_{0,\inf}-a_{2,\sup}\bar r_2^1-k\frac{\chi_1}{d_3}\bar r_1^1-(a_{1,\sup}-k\frac{\chi_1}{d_3})u-\tilde \epsilon\big({a_{2,\sup}+k\frac{\chi_1}{d_3}}\big)\big).
\end{equation*}
Thus similar arguments as those lead to \eqref{iter1-eq03} implies that there is
$t^1_{\epsilon,\underline u_0,\overline u_0,\overline v_0}\geq t^1_{\epsilon,\overline u_0,\overline v_0} $ such that
\begin{equation}\label{iter1-eq06}
  \underbar r_1^1-\epsilon \leq u(x,t;t_0,u_0,v_0)\,\, \forall \,\, t\geq t_0+t^1_{\epsilon,\underline u_0,\overline u_0,\overline v_0},\,\quad \forall\,\,  t_0 \in \mathbb{R}.
\end{equation}
Similarly, from \eqref{iter1-eq03}, \eqref{iter1-eq04}, the second and third equation of \eqref{u-v-w-eq00} and similar arguments as those lead to \eqref{iter1-eq03},
there is $ t^1_{\epsilon,\underline v_0,\overline u_0,\overline v_0}\geq t^1_{\epsilon,\overline v_0,\overline v_0} $ such that
\begin{equation}\label{iter1-eq07}
  \underbar r_2^1-\epsilon \leq v(x,t;t_0,u_0,v_0)\,\, \forall\,\,  t\geq t_0+t^1_{\epsilon,\underline u_0,\overline u_0,\overline v_0},\, \quad \forall \,\, t_0 \in \mathbb{R}.
\end{equation}
Choose $t^1_{\epsilon,\overline u_0,\overline v_0,\underline u_0,\underline v_0}=\max\{t_{\epsilon,\underline u_0,\overline u_0,\overline v_0},t_{\epsilon,\underline v_0,\overline u_0,\overline v_0}\} (\ge 0)$. Then  \eqref{iter1-eq00} follows from \eqref{iter1-eq03}, \eqref{iter1-eq04}, \eqref{iter1-eq06} and \eqref{iter1-eq07}.

\medskip

Next, assume that for any $\epsilon>0$, there is
 { $ t^k_{\epsilon,\overline u_0,\overline v_0,\underline u_0,\underline v_0}\geq t^{k-1}_{\epsilon,\overline u_0,\overline v_0,\underline u_0,\underline v_0}$} ($k\ge 2$)
 such that
\begin{equation}\label{iter2-eq000}
\begin{cases}
\underbar r_1^k -\epsilon \le u(x,t;t_0,u_0,v_0) \le \bar r_1^k+\epsilon\cr
\underbar r_2^k -\epsilon \le v(x,t;t_0,u_0,v_0) \le \bar r_2^k+\epsilon
\end{cases}
\end{equation}
for all $x\in\bar\Omega$, $t_0 \in \mathbb{R}$ and  { $t\ge t_0+t^k_{\epsilon,\overline u_0,\overline v_0,\underline u_0,\underline v_0}$}.
We claim that there is  there is
 { $ t^{k+1}_{\epsilon,\overline u_0,\overline v_0,\underline u_0,\underline v_0}\geq t^{k}_{\epsilon,\overline u_0,\overline v_0,\underline u_0,\underline v_0}$} ($k\ge 2$)
 such that
\begin{equation}\label{iter2-eq000-1}
\begin{cases}
\underbar r_1^{k+1} -\epsilon \le u(x,t;t_0,u_0,v_0) \le \bar r_1^{k+1}+\epsilon\cr
\underbar r_2^{k+1} -\epsilon \le v(x,t;t_0,u_0,v_0) \le \bar r_2^{k+1}+\epsilon
\end{cases}
\end{equation}
for all $x\in\bar\Omega$, $t_0 \in \mathbb{R}$ and  { $t\ge t_0+t^{k+1}_{\epsilon,\overline u_0,\overline v_0,\underline u_0,\underline v_0}$}.

In fact,
choose $0<\tilde \epsilon\le \epsilon$ such that
$$
\frac{a_{0,\sup}-a_{2,\inf}\underbar r_2^k-k\frac{\chi_1}{d_3}\underbar r_1^k{ -}\tilde \epsilon\big({ a_{2,\inf}+k\frac{\chi_1}{d_3}}\big)}{a_{1,\inf}-\frac{k\chi_1}{d_3}}+\tilde \epsilon\le \bar r_1^{k+1}+\epsilon
$$
and
$$
\frac{a_{0,\inf}-a_{2,\sup}\bar r_2^{k+1}-k\frac{\chi_1}{d_3}\bar r_1^{k+1}-\tilde \epsilon\big({ a_{2,\sup}+k\frac{\chi_1}{d_3}}\big)}{a_{1,\sup}-\frac{k\chi_1}{d_3}}-\tilde \epsilon\ge \underbar r_1^{k+1}-\epsilon.
$$
We have that for $t\ge t_0+t^k_{\tilde \epsilon,\overline u_0,\overline v_0,\underline u_0,\underline v_0}$,
\begin{equation*}
   u_t \le d_1\Delta u -\chi_1 \nabla w \cdot \nabla u +u\big(a_{0,\sup}-a_{2,\inf}\underbar r_2^k-k\frac{\chi_1}{d_3}\underbar r_1^k-(a_{1,\inf}-k\frac{\chi_1}{d_3})u-\tilde \epsilon\big({a_{2,\inf}+k\frac{\chi_1}{d_3}}\big)\big)
\end{equation*}
Then there is $\tilde t^{k+1}_{\epsilon,\overline u_0,\overline v_0,\underline u_0,\underline v_0}\ge  t^{k}_{\epsilon,\overline u_0,\overline v_0,\underline u_0,\underline v_0}$ such that for $t\ge \tilde t^{k+1}_{\epsilon,\overline u_0,\overline v_0,\underline u_0,\underline v_0}$
$$
u(x,t;t_0,u_0,v_0)\le \bar r^{k+1}+\epsilon\quad \forall\,\, x\in\bar\Omega,\,\, \forall\,\, t_0\in\mathbb{R}\,\, \forall \,\, t\ge t_0+\tilde t^{k+1}_{\epsilon,\overline u_0,\overline v_0,\underline u_0,\underline v_0}.
$$
This implies that for $t\ge t_0+\tilde t^{k+1}_{\epsilon,\overline u_0,\overline v_0,\underline u_0,\underline v_0}$,
\begin{equation*}
   u_t \geq d_1\Delta u -\chi_1 \nabla w \cdot \nabla u +u\big(a_{0,\inf}-a_{2,\sup}\bar r_2^{k+1}-k\frac{\chi_1}{d_3}\bar r_1^{k+1}-(a_{1,\sup}-k\frac{\chi_1}{d_3})u-\tilde \epsilon\big({ a_{2,\sup}+k\frac{\chi_1}{d_3}}\big)\big).
\end{equation*}
It then follows that there is $\bar t^{k+1}_{\epsilon,\overline u_0,\overline v_0,\underline u_0,\underline v_0}\ge  t^{k}_{\epsilon,\overline u_0,\overline v_0,\underline u_0,\underline v_0}$ such that for $t\ge t_0+\bar t^{k+1}_{\epsilon,\overline u_0,\overline v_0,\underline u_0,\underline v_0}$,
$$
\underbar r_1^{k+1} -\epsilon \le u(x,t;t_0,u_0,v_0) \le \bar r_1^{k+1}+\epsilon.
$$

Similarly, we can prove that there is  $\hat t^{k+1}_{\epsilon,\overline u_0,\overline v_0,\underline u_0,\underline v_0}\ge  t^{k}_{\epsilon,\overline u_0,\overline v_0,\underline u_0,\underline v_0}$ such that for $t\ge t_0+\hat  t^{k+1}_{\epsilon,\overline u_0,\overline v_0,\underline u_0,\underline v_0}$,
$$
\underbar r_2^{k+1} -\epsilon \le v(x,t;t_0,u_0,v_0) \le \bar r_2^{k+1}+\epsilon.
$$
The claim \eqref{iter2-eq000-1} then follows with
$t^{k+1}_{\epsilon,\overline u_0,\overline v_0,\underline u_0,\underline v_0}=\max\{\bar t^{k+1}_{\epsilon,\overline u_0,\overline v_0,\underline u_0,\underline v_0}, \hat t^{k+1}_{\epsilon,\overline u_0,\overline v_0,\underline u_0,\underline v_0}\}.$

Now, by induction, \eqref{itern-eq00a} holds for all $n\ge 1$. This completes the proof of (1).

\medskip

(2) It can be proved by the similar arguments as those in (1). We outline some idea in the following.

 First of all,  note that $\bar s_1^1=\bar s_1^0=\bar B_1$ and $\bar s_2^1=\bar s_2^0=\bar B_1$, and by {(H6)},
$0<\underbar s_1^1\le \bar s_1^1$ and $0<\underbar s_2^1\le\bar s_2^1$. \eqref{s-itern-eq05aa}
then follows from \eqref{s-iteration-eq} directly.

We  prove \eqref{s-itern-eq00a} by induction.

To this end,  we first   claim that there exists { $ t^1_{\epsilon,\overline u_0,\overline v_0,\underline u_0,\underline v_0}\ge 0$}
 such that
\begin{equation}\label{s-iter1-eq00}
\begin{cases}
\underbar s_1^1 -\epsilon \le u(x,t;t_0,u_0,v_0) \le \bar s_1^1+\epsilon\cr
\underbar s_2^1 -\epsilon \le v(x,t;t_0,u_0,v_0) \le \bar s_2^1+\epsilon
\end{cases}
\end{equation}
for all $x\in\bar\Omega$, $t_0 \in \mathbb{R}$ and { $t\ge t_0+t^1_{\epsilon,\overline u_0,\overline v_0,\underline u_0,\underline v_0}$}.

In fact, note that
 \begin{equation*}
\begin{cases}
u_t\le { d_1\Delta u -\chi_1 \nabla w \cdot \nabla u}+  u\big(a_{0,\sup}-{(a_{2,\inf}+l\frac{\chi_1}{d_3})} {\underbar s_2^0}-k\frac{\chi_1}{d_3}{\underbar s_1^0}-(a_{1,\inf}-k\frac{\chi_1}{d_3}) u+l\frac{\chi_1}{d_3}v\big) \\
v_t\le{ d_2\Delta v -\chi_2 \nabla w \cdot \nabla v}+  v\big(b_{0,\sup}-{ (b_{1,\inf}+k\frac{\chi_2}{d_3})} {\underbar s_1^0}-l\frac{\chi_2}{d_3}{\underbar s_2^0}-(b_{2,\inf}-l\frac{\chi_2}{d_3})v
+{ k \frac{\chi_2}{d_3}}u\big).
\end{cases}
\end{equation*}
Then for any $\epsilon>0$, there is  $\bar t^{1}_{\epsilon,\overline u_0,\overline v_0,\underline u_0,\underline v_0}\ge 0$ such that
$$
\begin{cases}
u(x,t;t_0,u_0,v_0)\le \bar s_1^1+\epsilon\cr
v(x,t;t_,u_0,v_0)\le \bar s_2^1+\epsilon
\end{cases}
$$
for all $x\in\Omega$, $t_0\in \mathbb{R}$, and $t\ge t_0+\bar t^{1}_{\epsilon,\overline u_0,\overline v_0,\underline u_0,\underline v_0}$.
This implies that for any $\tilde\epsilon>0$, $t\ge t_0+\bar t^{1}_{\tilde \epsilon,\overline u_0,\overline v_0,\underline u_0,\underline v_0}$,
\begin{equation*}
\begin{cases}
u_t\ge  { d_1\Delta u -\chi_1 \nabla w \cdot \nabla u}+u\big(a_{0,\inf}-{ (a_{2,\sup}+l\frac{\chi_1}{d_3})} {(\bar s_2^1+\tilde\epsilon)}-k\frac{\chi_1}{d_3}{(\bar s_1^1+\tilde\epsilon)}-(a_{1,\sup}-k\frac{\chi_1}{d_3}) u+l\frac{\chi_1}{d_3}v\big) \\
v_t\ge  { d_2\Delta v -\chi_2 \nabla w \cdot \nabla v}+ v\big(b_{0,\inf}-{ (b_{1,\sup}+k\frac{\chi_2}{d_3})} {(\bar s_1^1+\tilde\epsilon )}-l\frac{\chi_2}{d_3}{(\bar s_2^1+\tilde \epsilon)}-(b_{2,\sup}-l\frac{\chi_2}{d_3})v+{ k \frac{\chi_2}{d_3}}u\big).
\end{cases}
\end{equation*}
Choose $0<\tilde \epsilon<\epsilon$ such that
$$
\begin{cases}
\frac{\big(a_{0,\inf}-{ (a_{2,\sup}+l\frac{\chi_1}{d_3})}{(\bar s_2^1+\tilde\epsilon)}-k\frac{\chi_1}{d_3}{(\bar s_1^1+\tilde \epsilon)}\big)(b_{2,\sup}-l\frac{\chi_2}{d_3})}{(a_{1,\sup}-k\frac{\chi_1}{d_3})(b_{2,\sup}-l\frac{\chi_2}{d_3})-lk\frac{\chi_1\chi_2}
{d_3^2}}+\frac{\frac{l\chi_1}{d_3}\left(b_{0,\inf}-{ (b_{1,\sup}+k\frac{\chi_2}{d_3}) }{(\bar s_1^1+\tilde\epsilon)}-l\frac{\chi_2}{d_3}{
(\bar s_2^1+\tilde\epsilon)}\right)}{(a_{1,\sup}-k\frac{\chi_1}{d_3})(b_{2,\sup}-l\frac{\chi_2}{d_3})-lk\frac{\chi_1\chi_2}{d_3^2}}> \underbar s_1^1-\epsilon
\cr

\frac{\big(b_{0,\inf}-{ (b_{1,\sup}+k\frac{\chi_2}{d_3})}{(\bar s_1^1+\tilde\epsilon)}-l\frac{\chi_2}{d_3}{
(\bar s_2^1+\tilde\epsilon)}\big)(a_{1,\sup}-k\frac{\chi_1}{d_3})} {(a_{1,\sup}-k\frac{\chi_1}{d_3})(b_{2,\sup}-l\frac{\chi_2}{d_3})-lk\frac{\chi_1\chi_2}{d_3^2}}+\frac{\frac{k\chi_2}{d_3}\left(a_{0,\inf}
-{ (a_{2,\sup}+l\frac{\chi_1}{d_3}) } {(\bar s_2^1+\tilde\epsilon)}-k\frac{\chi_1}{d_3}{(\bar s_1^1+\tilde\epsilon)}\right)}{(a_{1,\sup}-k\frac{\chi_1}{d_3})(b_{2,\sup}-l\frac{\chi_2}{d_3})-lk\frac{\chi_1\chi_2}{d_3^2}}>\underbar s_2^1-\epsilon.
\end{cases}
$$
Then there is $ t^{1}_{\epsilon,\overline u_0,\overline v_0,\underline u_0,\underline v_0}\ge  \bar t^{1}_{\epsilon,\overline u_0,\overline v_0,\underline u_0,\underline v_0}$ such that
$$
\begin{cases}
u(x,t;t_0,u_0,v_0)\ge \underbar s_1^1-\epsilon\cr
v(x,t;t_,u_0,v_0)\ge \underbar s_2^1-\epsilon
\end{cases}
$$
 for $x\in\Omega$, $t_0\in\mathbb{R}$, and $t\ge t_0+t^{1}_{\epsilon,\overline u_0,\overline v_0,\underline u_0,\underline v_0}$.
 The claim \eqref{s-iter1-eq00} then follows.

Next,  assume that for any $\epsilon>0$, there is
 { $ t^k_{\epsilon,\overline u_0,\overline v_0,\underline u_0,\underline v_0}\geq t^{k-1}_{\epsilon,\overline u_0,\overline v_0,\underline u_0,\underline v_0}$} ($k\ge 2$)
 such that
\begin{equation}\label{iter2-eq000}
\begin{cases}
\underbar s_1^k -\epsilon \le u(x,t;t_0,u_0,v_0) \le \bar s_1^k+\epsilon\cr
\underbar s_2^k -\epsilon \le v(x,t;t_0,u_0,v_0) \le \bar s_2^k+\epsilon
\end{cases}
\end{equation}
for all $x\in\bar\Omega$, $t_0 \in \mathbb{R}$ and  { $t\ge t_0+t^k_{\epsilon,\overline u_0,\overline v_0,\underline u_0,\underline v_0}$}.
By the similar arguments as in (1),  there is  there is
 { $ t^{k+1}_{\epsilon,\overline u_0,\overline v_0,\underline u_0,\underline v_0}\geq t^{k}_{\epsilon,\overline u_0,\overline v_0,\underline u_0,\underline v_0}$} ($k\ge 2$)
 such that
\begin{equation}\label{iter2-eq000-1}
\begin{cases}
\underbar s_1^{k+1} -\epsilon \le u(x,t;t_0,u_0,v_0) \le \bar s_1^{k+1}+\epsilon\cr
\underbar s_2^{k+1} -\epsilon \le v(x,t;t_0,u_0,v_0) \le \bar s_2^{k+1}+\epsilon
\end{cases}
\end{equation}
for all $x\in\bar\Omega$, $t_0 \in \mathbb{R}$ and  { $t\ge t_0+t^{k+1}_{\epsilon,\overline u_0,\overline v_0,\underline u_0,\underline v_0}$}.

\eqref{s-itern-eq00a} then follows by induction and (2) is thus proved.
\end{proof}

\begin{lemma}\label{iteration-lm-01}
Consider \eqref{u-v-w-eq00}.
\begin{itemize}
\item[(1)]
Assume (H5). For  any given $n\in\NN$, $u_0,v_0 \in { C^+(\bar{\Omega})}$ with $u_0,v_0\not \equiv 0$ and $t_0 \in \mathbb{R}$, if
\begin{equation}\label{iteration-lm-01-eq00}
  \underbar r_1^n\leq u_0\leq \bar r_1^n \,\,\text{and}\,\, \underbar r_2^n\leq v_0\leq \bar r_2^n,
\end{equation}
then
\begin{equation}\label{iteration-lm-01-eq01}
  \underbar r_1^n\leq u(x,t;t_0,u_0,v_0)\leq \bar r_1^n \,\,\text{and}\,\, \underbar r_2^n\leq v(x,t;t_0,u_0,v_0)\leq \bar r_2^n \,\, \forall \,\,  t\geq t_0.
\end{equation}

\item[(2)] Assume (H6). For any given $n\in\NN$,  $u_0,v_0 \in { C^+(\bar{\Omega})}$ with  $u_0,v_0\not \equiv 0$ and $t_0 \in \mathbb{R}$, if
\begin{equation}\label{s-iteration-lm-01-eq00}
  \underbar s_1^n\leq u_0\leq \bar s_1^n \,\,\text{and}\,\, \underbar s_2^n\leq v_0\leq \bar s_2^n,
\end{equation}
then
\begin{equation}\label{s-iteration-lm-01-eq01}
  \underbar s_1^n\leq u(x,t;t_0,u_0,v_0)\leq \bar s_1^n \,\,\text{and}\,\, \underbar s_2^n\leq v(x,t;t_0,u_0,v_0)\leq \bar s_2^n \,\, \forall\,\,  t\geq t_0.
\end{equation}
\end{itemize}
\end{lemma}

\begin{proof}
(1)  For given $n\in\NN$, suppose \eqref{iteration-lm-01-eq00} holds.  We prove \eqref{iteration-lm-01-eq01} holds  in two steps.

\medskip
\noindent {\bf Step 1. } We prove in this step
that the following holds for $k=1$,
\begin{equation}\label{iteration-lm-01-eq01-0}
  \underbar r_1^k\leq u(x,t;t_0,u_0,v_0)\leq \bar r_1^k \,\,\text{and}\,\, \underbar r_2^k\leq v(x,t;t_0,u_0,v_0)\leq \bar r_2^k \,\, \,\, \forall \,\,  t\geq t_0.
\end{equation}
Recall that \eqref{iter1-eq01} reads as
$$
   u_t \leq d_1\Delta u -\chi_1 \nabla w \cdot \nabla u +u\big(a_{0,\sup}-(a_{1,\inf}-k\frac{\chi_1}{d_3})u\big).
$$
Thus, by parabolic comparison principe and $\bar u_0 \leq \bar r_1^n\le \bar r_1^1$, we get  that
\begin{equation}\label{step1-optimal-eq00}
0\leq u(x,t;t_0,u_0,v_0)\leq \bar r_1^1 \quad \forall\,\,  t\geq t_0.
\end{equation}
Similarly, by parabolic comparison principe and  $\bar v_0 \leq \bar r_2^n\le  \bar r_2^1$, we can get that
\begin{equation}\label{step1-optimal-eq01}
0\leq v(x,t;t_0,u_0,v_0)\leq \bar r_2^1 \quad \forall\,\, t\geq t_0.
\end{equation}
Therefore, for $t\ge t_0$,
$$
u_t \geq d_1\Delta u -\chi_1 \nabla w \cdot \nabla u +u\big(a_{0,\inf}-a_{2,sup}\bar r_2^1-k\frac{\chi_1}{d_3}\bar r_1^1-(a_{1,\sup}-k\frac{\chi_1}{d_3})u\big)\big).
$$
By parabolic comparison principe and $\underbar r_1^1\le \underbar r_1^n \leq \underbar u_0$, we have that
\begin{equation}\label{step1-optimal-eq02}
\underbar r_1^1\leq u(x,t;t_0,u_0,v_0) \quad \forall \,\, t\geq t_0.
\end{equation}
Similarly, by parabolic comparison principe and $\underbar r_2^1\le \underbar r_2^n\leq \underbar v_0 $, we have that
\begin{equation}\label{step1-optimal-eq03}
\underbar r_2^1\leq v(x,t;t_0,u_0,v_0) \quad \forall \,\, t\geq t_0.
\end{equation}
Thus the result follows from \eqref{step1-optimal-eq00},\eqref{step1-optimal-eq02},\eqref{step1-optimal-eq01} and \eqref{step1-optimal-eq03}.

\medskip
\noindent {\bf Step 2.} Suppose that \eqref{iteration-lm-01-eq01-0} holds for $k=1,2,\cdots,l$ ($l\le n-1$), we prove that \eqref{iteration-lm-01-eq01-0} holds for $k=l+1$.

Indeed since \eqref{iteration-lm-01-eq01-0} holds  for $ 1\leq k\leq l,$ for $t\geq t_0,$ we get from the first and third equation of \eqref{u-v-w-eq00}  that
$$
   u_t \leq d_1\Delta u -\chi_1 \nabla w \cdot \nabla u +u\big(a_{0,\sup}-a_{2,\inf}\underbar r_2^l-k\frac{\chi_1}{d_3}\underbar r_1^l-(a_{1,\inf}-k\frac{\chi_1}{d_3})u\big).
$$
Thus,  by parabolic comparison principe and $ \bar u_0 \leq \bar r_1^{n}\le \bar r_1^{l+1}$, we get that
\begin{equation}\label{stepn-optimal-eq00}
 u(x,t;t_0,u_0,v_0) \leq \bar r_1^{l+1} \quad \forall t\geq t_0.
\end{equation}
Similarly, from the second and third equation of \eqref{u-v-w-eq00} and parabolic comparison principe, we get since $ \bar v_0 \leq \bar r_2^{n}\le \bar r_2^{l+1}$ that
\begin{equation}\label{stepn-optimal-eq01}
 v(x,t;t_0,u_0,v_0) \leq \bar r_2^{l+1} \quad \forall t\geq t_0.
\end{equation}
Next again from the first and third equation of \eqref{u-v-w-eq00} that
$$
   u_t \geq d_1\Delta u -\chi_1 \nabla w \cdot \nabla u +u\big(a_{0,\inf}-a_{2,sup}\bar r_2^l-k\frac{\chi_1}{d_3}\bar r_1^l-(a_{1,\sup}-k\frac{\chi_1}{d_3})u\big).
$$
Therefore by parabolic comparison principe we get since $\underbar r_1^{l+1}\le \underbar r_1^{n}\leq  \underbar u_0 $ that
\begin{equation}\label{stepn-optimal-eq02}
\underbar r_1^{l+1}\leq  u(x,t;t_0,u_0,v_0)  \quad \forall t\geq t_0.
\end{equation}
Similarly, from the second and third equation of \eqref{u-v-w-eq00} and parabolic comparison principe, we get since $\underbar r_2^{l+1}\le \underbar r_2^{n}\leq  \underbar u_0 $  that
\begin{equation}\label{stepn-optimal-eq03}
 \underbar r_2^{l+1} \leq v(x,t;t_0,u_0,v_0)  \quad \forall t\geq t_0.
\end{equation}
By \eqref{stepn-optimal-eq00}, \eqref{stepn-optimal-eq02}, \eqref{stepn-optimal-eq01} and \eqref{stepn-optimal-eq03}, \eqref{iteration-lm-01-eq01-0} holds for $k=l+1$.

\eqref{iteration-lm-01-eq01} then follows by induction.

(2) It can be proved by the similar arguments as those in (1).
\end{proof}

Now we prove Theorem \ref{attracting-rectangle-thm}.

\begin{proof} [Proof of Theorem \ref{attracting-rectangle-thm}]

(1)  First of all, from \eqref{itern-eq05aa}, the sequences $\underbar r_1^n$ and $\underbar r_2^n$ are nondecreasing bounded sequences of nonnegative real numbers and the sequences $\bar r_1^n$ and $\bar r_2^n$ non-increasing bounded sequences of nonnegative real numbers. Thus there exist real numbers $0< \underbar r_1\leq \bar r_1\le \bar A_1 $ and $0< \underbar r_2\leq \bar r_2\le\bar A_2$ such that
\begin{equation}\label{lim-eq000}
  \begin{cases}
  \lim_{n\to \infty}\underbar r_1^n=\underbar r_1,\quad \lim_{n\to \infty}\bar r_1^n=\bar r_1, \cr
  \lim_{n\to \infty}\underbar r_2^n=\underbar r_2, \quad \lim_{n\to \infty}\bar r_2^n=\bar r_2.
  \end{cases}
\end{equation}
 Combining \eqref{lim-eq000} with \eqref{r-iteration-eq}, we get
\begin{equation}\label{lim-def00}
  \bar r_1=\frac{a_{0,\sup}-a_{2,\inf}\underbar r_2-k\frac{\chi_1}{d_3}\underbar r_1}{a_{1,\inf}-k\frac{\chi_1}{d_3}},
\end{equation}
\begin{equation}\label{lim-def01}
 \bar r_2=\frac{b_{0,\sup}-b_{1,\inf}\underbar r_1-k\frac{\chi_1}{d_3}\underbar r_2}{b_{2,\inf}-l\frac{\chi_2}{d_3}},
\end{equation}
\begin{equation}\label{lim-def03}
\underbar r_1=\frac{a_{0,\inf}-a_{2,\sup}\bar r_2-k\frac{\chi_1}{d_3}\bar r_1}{a_{1,\sup}-k\frac{\chi_1}{d_3}},
\end{equation}
and
\begin{equation}\label{lim-def04}
\underbar r_2=\frac{b_{0,\inf}-b_{1,\sup}\bar r_1-l\frac{\chi_2}{d_3}\bar r_1}{b_{2,\sup}-l\frac{\chi_2}{d_3}}.
\end{equation}
Hence   $(\bar r_1,\bar r_2,\underbar r_1,\underbar r_2)$ is the unique  solution of
 \eqref{attracting-rectabgle-eq1}.

Next, we  prove \eqref{attracting-set-eq00}.  By  \eqref{itern-eq00a} and \eqref{lim-eq000}, for any $\epsilon>0$, we can choose $N$  such
\begin{equation*}\label{itern-eq000}
\begin{cases}
\underbar r_1-2\epsilon\leq \underbar r_1^N -\epsilon \le u(x,t;t_0,u_0,v_0) \le \bar r_1^N+\epsilon \le \bar r_1+2\epsilon \cr
\underbar r_2-2\epsilon\leq\underbar r_2^N -\epsilon \le v(x,t;t_0,u_0,v_0) \le \bar r_2^N+\epsilon \le \bar r_2+2\epsilon,
\end{cases}
\end{equation*}
for all $x\in\bar\Omega$, $t_0 \in \mathbb{R}$ and { $t\ge t_0+t^N_{\epsilon,\overline u_0,\overline v_0,\underline u_0,\underline v_0}$.} Thus \eqref{attracting-set-eq00} holds.

Now suppose that \eqref{invariant-set-eq00} holds. We prove \eqref{invariant-set-eq01}. Assume that
$$
  \underbar r_1\leq u_0\leq \bar r_1 \,\,\text{and}\,\, \underbar r_2\leq v_0\leq \bar r_2.
$$
Since the sequences $\underbar r_1^n$ and $\underbar r_2^n$ are nondecreasing bounded sequences of nonnegative real numbers and the sequences $\bar r_1^n$ and $\bar r_2^n$ non-increasing bounded sequences of nonnegative real numbers, from \eqref{lim-eq000}, we get for $n \in \mathbb{N}$ that
$$
 \underbar r_1^n\leq  \underbar r_1\leq u_0\leq \bar r_1\leq \bar r_1^n \,\,\text{and}\,\, \underbar r_2^n\leq \underbar r_2\leq v_0\leq \bar r_2\leq \bar r_2^n.
$$
By Lemma \ref{iteration-lm-01},
$$
\underbar r_1^n\leq u(x,t;t_0,u_0,v_0)\leq \bar r_1^n  \,\,\text{and}\,\,\underbar r_2^n\leq v(x,t;t_0,u_0,v_0)\leq \bar r_2^n \,\, \forall n \in \mathbb{N}\,\,, t\geq t_0.
$$
Then as $n \to \infty,$ we get
$$
  \underbar r_1\leq u(x,t;t_0,u_0,v_0)\leq \bar r_1 \,\,\text{and}\,\, \underbar r_2\leq v(x,t;t_0,u_0,v_0)\leq \bar r_2 \,\, \forall t\geq t_0.
$$
Thus \eqref{invariant-set-eq01} holds.

(2) It follows from the similar arguments as those in  (1).
\end{proof}

\section{Uniqueness and stability of coexistence states and Proof of Theorem \ref{thm-nonlinear-stability-001}}

In this section, we establish the nonlinear stability and uniqueness of entire solutions of system  \eqref{u-v-w-eq00} and prove Theorem \ref{thm-nonlinear-stability-001} and Corollary \ref{stability-cor}.

We first prove Theorem \ref{thm-nonlinear-stability-001}(3).

\begin{proof} [Proof of Theorem \ref{thm-nonlinear-stability-001}(3)]

Recall that \eqref{stability-cond-1-eq2} implies (H2) (see Remark \ref{rmk-stability}(2)).  For given $t_0\in\RR$ and $u_0,v_0\in {C^+(\bar{\Omega})}$ with  $u_0(\cdot),v_0(\cdot)\not =0$, let $(u(\cdot,t;t_0,u_0,v_0),v(\cdot,t;t_0,u_0,v_0)$, $w(\cdot,t;t_0,u_0,v_0))$ be
the solution of \eqref{u-v-w-eq00} given by Theorem \ref{thm-global-000}(2).  Note that  $(u(\cdot,t;t_0,u_0,v_0)$, $v(\cdot,t;t_0,u_0,v_0)$, $w(\cdot,t;t_0,u_0,v_0))$ exists for all $t>t_0$ and without loss of generality, we may assume that
$u_0(x),v_0(x)>0$ for all $x\in\bar\Omega$.

Let $(u^{**}(t),v^{**}(t),w^{**}(t))$ be a spatially homogeneous coexistence state of \eqref{u-v-w-eq00} (see Remark \ref{rmk-stability}(1)). We first prove that
\eqref{global-stability-1-eq1} and \eqref{global-stability-1-eq2}
hold.

To this end, let $(\overline u(t), \underline u(t), \overline v(t), \underline v(t))$ be as in { Lemma \ref{lem-1-ode000}}. Then
by Lemma \ref{lem-1-ode000}, we have
\begin{equation}
\label{thm-4-eq1}
\underline u(t)\le u(x,t;t_0,u_0,v_0)\le \bar u(t),\quad \underline v(t)\le v(x,t;t_0,u_0,v_0)\le \bar v(t)\quad \, \forall\,\, x\in\bar\Omega,\,\, t\ge t_0.
\end{equation}
We claim that for any $\epsilon>0$, there is $t_{\epsilon,{ u_0,v_0,t_0}}>0$  such that
 \begin{equation}
  \label{thm-4-eq2}
  \underline u(t)-\epsilon\le u^{**}(t)\le \bar u(t)+\epsilon,\quad \underline v(t)-\epsilon\le v^{**}(t)\le \bar v(t)+\epsilon\quad \forall\,\, t\ge t_0+t_{\epsilon,{ u_0,v_0,t_0}}.
 \end{equation}
{Indeed  let $(u^1(t),v^1(t))$ be the solution of \eqref{u-v-ode} with $(u^1(t_0),v^1(t_0))={(\underline u_0,\overline v_0)}.$
 Note
that  $(\overline u(t), \underline v(t))$ satisfies
\begin{equation}
\label{thm-4-ode001}
\begin{cases}
\overline u_t\geq \overline u(t)( a_0(t)-a_1(t)\overline u(t) -a_2(t)\underline v(t))\\
\underline v_t\leq \underline v(t)( b_0(t)-b_1(t)\overline u(t) -b_2(t)\underline v(t)).
\end{cases}
\end{equation}
Then by comparison principle for two species competition systems,
\begin{equation}
\label{aaux-stability-eq1}
u^1(t)\le \overline u(t) \quad {\rm and}\quad v^1(t)\ge \underline v(t)\quad \text{for all} \quad t\geq t_0.
\end{equation}
 Similarly, let $(u^2(t),v^2(t))$ be the solution of \eqref{u-v-ode} with $(u^2(t_0),v^2(t_0))={(\overline u_0,\underline v_0)}.$  Note that
\begin{equation}
\label{thm-4-ode002}
\begin{cases}
\underline u_t\leq \underline u(t)( a_0(t)-a_1(t)\underline u(t) -a_2(t)\overline v(t))\\
\overline v_t\geq \overline v(t)( b_0(t)-b_1(t)\underline u(t) -b_2(t)\overline v(t))\\
\end{cases}
\end{equation}
By comparison principle for two species competition systems again,
\begin{equation}
\label{aaux-stability-eq2}
u^2(t)\ge \underline u(t)\quad {\rm and}\quad  v^2(t)\le \overline v(t)\quad \text{for all}\quad t\geq t_0.
\end{equation}
By Lemma \ref{persistence-lm6},
\begin{equation*}
\lim_{t\to\infty}\big(|u^i(t)-u^{**}(t)|+|v^i(t)-v^{**}(t)|\big)=0\quad {\rm for}\quad i=1,2.
\end{equation*}
This implies that for any $\epsilon>0$,  there is $t_{\epsilon,{ u_0,v_0,t_0}}>0$  such that
 \begin{equation}
  \label{aaux-stability-eq3}
  u^2(t)-\epsilon\le u^{**}(t)\le u^1(t)+\epsilon,\quad  v^1(t)-\epsilon\le v^{**}(t)\le v^2(t)+\epsilon\quad \forall\,\, t\ge t_0+t_{\epsilon,{ u_0,v_0,t_0}}.
 \end{equation}
\eqref{thm-4-eq2} then follows from \eqref{aaux-stability-eq1}, \eqref{aaux-stability-eq2}, and \eqref{aaux-stability-eq3}.
}

By \eqref{thm-4-eq1}  and \eqref{thm-4-eq2}, to show \eqref{global-stability-1-eq1} and \eqref{global-stability-1-eq2},
 it suffices to show $0 \leq \ln \frac{\overline{u}(t)}{\underline{u}(t)}+ \ln \frac{\overline{v}(t)}{\underline{v}(t)}\longrightarrow 0  \, \text{as } t \to \infty.$ Assume  that $t>t_0$.  By \eqref{ode00}, we have
{\large
\begin{equation*}
\begin{cases}
\frac{\overline{u}'}{\overline{u}}=\frac{\chi_1}{d_3} \big(k \overline {u}+l\overline v-k\underline{u}-l\underline{v}\big)+ \big[a_0(t)-a_1(t)\overline u
-a_2(t)\underline{v}\big]\\
\frac{\underline{u}'}{\underline{u}}=\frac{\chi_1}{d_3} \big(k \underline {u}+l\underline v-k\overline{u}-l\overline{v}\big)+ \big[a_0(t)-a_1(t)\underline u
-a_2(t)\overline{v}\big]\\
\frac{\overline{v}'}{\overline{v}}=\frac{\chi_2}{d_3} \big(k \overline {u}+l\overline v-k\underline{u}-l\underline{v}\big)+ \big[b_0(t)-b_2(t)\overline v
-b_1(t)\underline{u}\big]\\
\frac{\underline{v}'}{\underline{v}}=\frac{\chi_2}{d_3}\big(k \underline {u}+l\underline v-k\overline{u}-l\overline{v}\big)+ \big[b_0(t)-b_2(t)\underline v
-b_1(t)\overline{u}\big].
\end{cases}
\end{equation*}}
This together with \eqref{stability-cond-1-eq1}  implies that
 \begin{equation}
 \label{aaux-stability-eq4}
 \frac{d}{dt} \Big(       \ln \frac{\overline{u}}{\underline{u}} +   \ln \frac{\overline{v}}{\underline{v}}   \Big)=
\frac{\overline{u}'}{\overline{u}} -   \frac{\underline{u}'}{\underline{u}}+\frac{\overline{v}'}{\overline{v}} -   \frac{\underline{v}'}{\underline{v}}\leq -\min\{\alpha_1,\beta_1\}   \left(  (\overline {u}-\underline{u})+(\overline {v}-\underline{v}) \right)\leq 0,
\end{equation}
where
$$ 0< \alpha_1=\inf_{t\in \RR}\{a_1(t)-b_1(t)-2k\frac{\chi_1+\chi_2}{d_3}\},$$ and $$ 0< \beta_1=\inf_{t \in \RR}\{b_2(t)-a_2(t)-2l\frac{\chi_1+\chi_2}{d_3}\}.$$
Thus by integrating \eqref{aaux-stability-eq4} over $(t_0,t),$ we get
$$
   0 \leq    \ln \frac{\overline{u}(t)}{\underline{u}(t)}+\ln \frac{\overline{v}(t)}{\underline{v}(t)} \leq \ln \frac{\overline{u}_0}{\underline{u}_0}+\ln \frac{\overline{v}_0}{\underline{v}_0}, \quad{\rm and\,\,\,  then}\quad  \frac{\overline{u}(t)\overline{v}(t)}{\underline{u}(t)\underline{v}(t)}  \leq  \frac{\overline{u}_0 \overline{v}_0}{\underline{u}_0 \underline{v}_0}.
$$
 We have by mean value theorem that

  $$-\left(  (\overline {u}-\underline{u})+(\overline {v}-\underline{v}) \right)\leq -\underline{u} \Big(     \ln \frac{\overline{u}}{\underline{u}} \Big) -\underline{v} \Big(     \ln \frac{\overline{v}}{\underline{v}} \Big)
 $$
Therefore
 \begin{equation}
 \label{aaux-stability-eq5}
 \frac{d}{dt} \Big(       \ln \frac{\overline{u}}{\underline{u}} +   \ln \frac{\overline{v}}{\underline{v}}   \Big) \leq      -\Big( \min\{\alpha_1,\beta_1\} \Big)  \Big(  \min\{\alpha_2,\beta_2\}  \Big)        \Big(     \ln \frac{\overline{u}}{\underline{u}}+ \ln \frac{\overline{v}}{\underline{v}} \Big),
 \end{equation}
where
$$ 0<\alpha_2:=\alpha_{2,{ t_0,u_0,v_0}}=\inf_{t\ge t_0} \overline{u}(t)   \frac{\underline{u}_0\underline{v}_0}{\overline{u}_0\overline{v}_0},$$ and $$ 0<\beta_2:=\beta_{2,{ t_0,u_0,v_0}}=\inf_{t\ge t_0} \overline{v}(t)   \frac{\underline{u}_0\underline{v}_0}{\overline{u}_0\overline{v}_0}.$$
By letting $\epsilon_{0,{ t_0,u_0,v_0}}=\big( \min\{\alpha_1,\beta_1\} \big)  \big(  \min\{\alpha_2,\beta_2\}  \big),$
we have $\epsilon_{0,{ t_0,u_0,v_0}}>0$
and
 \[   0\leq      \ln \frac{\overline{u}}{\underline{u}}+\ln \frac{\overline{v}}{\underline{v}} \leq   \left(\ln \frac{\overline{u}_0}{\underline{u}_0}+\ln \frac{\overline{v}_0}{\underline{v}_0}\right) e^{-\epsilon_{0,{ t_0,u_0,v_0}} (t-t_0)}   \to 0\quad {\rm as}\quad t\to\infty.
  \]
Hence \eqref{global-stability-1-eq1} and \eqref{global-stability-1-eq2} hold.

Next, we show that \eqref{u-v-w-eq00} has a unique spatially homogeneous coexistence state.  Suppose that $(u^{*}_i(t),v^{*}_i(t),w^{*}_i(t))$ ($i=1,2$)
are spatially homogeneous coexistence states of \eqref{u-v-w-eq00}. { Let
$u_{01}=\max\{\sup_{t\in\RR}u_1^*(t), \sup_{t\in\RR}u_2^*(t)\}$, $v_{01}=\min\{\inf_{t\in\RR}v_1^*(t),\inf_{t\in\RR}v_2^*(t)\}$,
$u_{02}=\min\{\inf_{t\in\RR}u_1^*(t)$, $\inf_{t\in\RR}u_2^*(t)\}$, and $v_{02}=\max\{\sup_{t\in\RR}v_1^*(t),\sup_{t\in\RR}v_2^*(t)\}$.}
For any $t_0\in\RR$, let $(u_i(t),v_i(t))=(u(t;t_0,u_{0i},v_{0i}),v(t;t_0,u_{0i},v_{0i}))$ be the solution of \eqref{u-v-ode} with
$(u(t_0;t_0,u_{0i},v_{0i}),v(t_0;t_0,u_{0i},v_{0i}))=(u_{0i},v_{0i})$ ($i=1,2$). By comparison principle for two species competition systems,
\begin{equation}
\label{aaux-stability-eq6}
u_2(t)\le u_i^*(t)\le u_1(t)\quad {\rm and}\quad
v_1(t)\le v_i^*(t)\le v_2(t)
\end{equation}
for $i=1,2$ and $t\ge t_0$.
{By the definition of coexistence states}, there are $0<\delta<K$  such that
\begin{equation}
\label{aaux-stability-eq7}
\delta\le u_i^{ *}(t)\le K,\quad \delta \le v_i^{ *}(t)\le K
\end{equation}
for $i=1,2$ and all $t\in\RR$. By the similar  arguments of \eqref{aaux-stability-eq4}, we have
$$
\frac{d}{dt}\ln \frac{u_1(t)}{u_2(t)}+\frac{d}{dt}\ln \frac{v_2(t)}{v_1(t)}\le -\min\{\tilde \alpha_1,\tilde \beta_1\}\big(u_1(t)-u_2(t)+ v_{ 2}(t)-v_{ 1}(t)\big)
$$
for $t\ge t_0$, where
$$
\tilde \alpha_1=\inf_{t\in\RR}(a_1(t)-b_1(t)),\quad \tilde \beta_1=\inf_{t\in\RR}( b_2(t)-a_2(t)).
$$
Let
$$ \tilde \alpha_2={ \inf_{t\in \RR}} {u_1^{ *}}(t)   \frac{u_{02}v_{01}}{u_{01}v_{02}},
\quad  \tilde \beta_2={ \inf_{t\in\RR}} v_2^{ *}(t)   \frac{u_{02}v_{01}}{u_{01}v_{02}}$$
and
 $\tilde \epsilon_0=\big( \min\{\tilde \alpha_1,\tilde \beta_1\} \big)  \big(  \min\{\tilde \alpha_2,\tilde \beta_2\}  \big).$
 Then by the similar arguments of \eqref{aaux-stability-eq5}, we have
 $$
  0\leq      \ln \frac{u_1(t+t_0)}{u_2(t+t_0)}+\ln \frac{v_2(t+t_0)}{v_1(t+t_0)} \leq   \left(\ln \frac{u_{01}}{u_{02}}+\ln \frac{v_{02}}{v_{01}}\right) e^{-\tilde \epsilon_0 t}
  $$
  for $t\ge t_0$. This together with \eqref{aaux-stability-eq7} implies that
  {
  $$
  0\leq      \ln \frac{u_1(t+t_0)}{u_2(t+t_0)}+\ln \frac{v_2(t+t_0)}{v_1(t+t_0)} \leq   2\ln \Big(\frac{K}{\delta}\Big)e^{-\tilde \epsilon_0 t}.
  $$
  }
  Therefore
  $$
  \lim_{t\to\infty}\ln \frac{u(t+t_0;t_0;u_{01},v_{01})}{u(t+t_0;t_0,u_{02},v_{02})}+\ln  \frac{v(t+t_0;t_0;u_{02},v_{02})}{v(t+t_0;t_0,u_{01},v_{01})}=0
  $$
  uniformly in $t_0\in\RR$.
  It then follows from \eqref{aaux-stability-eq6} that $u_1^*(t)\equiv u_2^*(t)$ and $v_1^*(t)\equiv v_2^*(t)$.
  { Indeed let $t \in \mathbb{R}$ be given. It follows from \eqref{aaux-stability-eq6} that
  \begin{align*}
    |u_1^*(t)-u_2^*(t)| & =|\tilde u^*(t)\ln \Big(\frac{u_1^*(t)}{u_2^*(t)}\Big)| \quad \text{(for some}\,\, \tilde u^*(t)\,\, \text{between}\,\, u_1^*(t)\,\, {\rm and}\,\, u_2^*(t))\\
     & \leq \max\{|u_1^*(t)|,|u_2^*(t)|\} |\ln \Big(\frac{u_1^*(t)}{u_2^*(t)}\Big)|\\
     & \leq K \ln\Big(\frac{u_1(t)}{u_2(t)}\Big) \\
     & \leq 2K\ln \Big(\frac{K}{\delta}\Big)e^{-\tilde \epsilon_0 (t-t_0)}, \,\, \forall t_0\leq t.
  \end{align*}
 And similarly
 $$
 |v_1^*(t)-v_2^*(t)|\leq 2K \ln \Big(\frac{K}{\delta}\Big)e^{-\tilde \epsilon_0 (t-t_0)}, \,\, \forall t_0\leq t.
 $$
 Therefore as $t_0 \to -\infty,$ we get $|u_1^*(t)-u_2^*(t)|=|v_1^*(t)-v_2^*(t)|{ =0 }.$
  } Hence \eqref{u-v-w-eq00}
  has a unique spatially homogeneous coexistence state.
  \end{proof}

  Next, we prove Theorem \ref{thm-nonlinear-stability-001}(1) and (2).

  \begin{proof}[Proof of Theorem \ref{thm-nonlinear-stability-001}(1) and (2)]
 { For given $u_0,v_0 \in {C^+(\bar{\Omega})},$  let $\overline{u}_0=\max_{x \in \bar \Omega}u_0(x),$ $\underline{u}_0=\min_{x \in \bar \Omega}u_0(x),$ $\overline{v}_0=\max_{x \in \bar \Omega}v_0(x)$ , $\underline{v}_0=\min_{x \in \bar \Omega}v_0(x)$.

 (1) By Theorem 1.2(1) and Remark 1.3(2), \eqref{u-v-w-eq00} has coexistence states.
 Let $(u^{**}(x,t)$, $v^{**}(x,t),w^{**}(x,t))$ be a coexistence state of \eqref{u-v-w-eq00}.}
 Let $q_1(t),$ $Q_1(t),$ $q_2(t)$ and $Q_2(t)$ be as in \eqref{q-eq1}, \eqref{q-eq2}, \eqref{q-eq3}  and \eqref{q-eq4}, respectively. By \eqref{stability-cond-2-eq1},
 $$\mu=\limsup_{t-s\to\infty}\frac{1}{t-s}\int_{s}^{t}\max\{q_1(\tau)- Q_1(\tau),q_2(\tau)- Q_2(\tau)\}d\tau  <0.
$$
 Fix $0<\epsilon<-\mu.$
 Then, for given  $u_0,v_0 \in C^0(\bar \Omega)$ with
  $\inf u_0>0$, $\inf v_0>0$, there exists   $T_{\epsilon,u_0,v_0}:= T_{\epsilon,\bar u_0,\bar v_0,\underline{u}_0, \underline{v}_0}>0$ such that for any $t_0\in\mathbb{R}$,
$$\underbar r_1-\epsilon \le u(\cdot,t_0+t;t_0;u_0,v_0)\le {\bar r_1}+\epsilon, \, \,\, \underbar r_1-\epsilon \le u^{**}(x,t) \le {\bar r_1}+\epsilon\,\, \,  \forall x\in\bar\Omega,\, \,   t\ge t_0+T_{\epsilon,u_0,v_0},
$$
$$\underbar r_2-\epsilon \le v(\cdot,t_0+t;t_0;u_0,v_0)\le {\bar r_2}+\epsilon, \, \,\, \underbar r_2-\epsilon \le v^{**}(x,t) \le {\bar r_2}+\epsilon\,\, \,  \forall x\in\bar\Omega,\, \,  t\ge t_0+T_{\epsilon,u_0,v_0},
$$
and
$$\int_{t_0}^{t_0+t}\max\{q_1(s)- Q_1(s),q_2(s)- Q_2(s)\}ds \le (\mu_1+\epsilon)t, \,  \forall\,\,  t\ge t_0+T_{\epsilon,u_0,v_0}.
$$

To simplify the notation,
set $u(t)=u(\cdot,t;t_0;u_0,v_0),$ $v(t)=v(\cdot,t;t_0;u_0,v_0),$ $u^{**}(t)=u^{**}(\cdot,t),$ and $v^{**}(t)=v^{**}(\cdot,t).$ Let $\psi=u-u^{**}$ and $\phi=v-v^{**}.$ Then $\psi$ satisfies
\begin{align}
\label{global-eq00}
\psi_t &=d_1\Delta \psi-\chi_1\nabla \cdot(\psi\nabla w)-\chi_1\nabla \cdot(u^{**}\nabla(w-w^{**})) \nonumber\\
&+ \psi\Big(a_0(t,x)-a_1(t,x)(u+u^{**})-a_2(t,x)v\Big)-a_2(t,x)u^{**}\phi,
\end{align}
and $\phi$ satisfies
\begin{align}
\label{global-eq00-1}
\phi_t &=d_2\Delta \phi-\chi_2\nabla \cdot(\phi\nabla w)-\chi_2\nabla \cdot(v^{**}\nabla(w-w^{**})) \nonumber\\
&+ \psi\Big(b_0(t,x)-b_1(t,x)u-b_2(t,x)(v+v^{**})\Big)-b_1(t,x)v^{**}\psi.
\end{align}

We first prove that $\int_{\Omega}\big(\psi^2+\phi^2\big)dx\to 0$ as $t\to\infty$ uniformly in $t_0\in\RR$.
To this end,
by multiplying \eqref{global-eq00} by $\psi_+$ and integrating over $\Omega,$ we get
\begin{multline*}
\frac{1}{2}\frac{d}{dt}\int_{\Omega}\psi_+^2+d_1\int_{\Omega}|\nabla \psi_+|^2=\chi_1\int_{\Omega}\psi_+\nabla \psi_+ \cdot \nabla  w+\chi_1\int_{\Omega}u^{**}\nabla \psi_{+} \cdot \nabla (w-w^{**})\\
+\int_{\Omega} \psi_+^2\big(a_0(t,x)-a_1(t,x)(u+u^{**})-a_2(t,x)v\big)-\int_{\Omega}a_2(t,x)u^{**}\psi_{+}\phi
\end{multline*}
for  a.e  $t>t_0$ (see \cite[(4.6)]{ITBWS16} for the reasons to have the above equality).
Then by integrating by parts, we get
\begin{align}
\label{new-aux-eq1}
\frac{1}{2}\frac{d}{dt}\int_{\Omega}\psi_+^2+d_1\int_{\Omega}|\nabla \psi_+|^2\le &-\frac{\chi_1}{2}\int_{\Omega}\psi_+^2\Delta w+\chi_1\int_{\Omega}u^{**}\nabla \psi_{+} \cdot \nabla (w-w^{**})\nonumber\\
&+\int_{\Omega} \psi_+^2\big(a_{0,\sup}(t)-a_{1,\inf}(t)(u+u^{**})-a_{2,\inf}(t)v\big)\nonumber\\
&-\int_{\Omega}a_2(t,x)u^{**}\psi_{+}\phi_++\int_{\Omega}a_2(t,x)u^{**}\psi_{+}\phi_-
\end{align}
for  a.e  $t>t_0$.

We have from the third equation of \eqref{u-v-w-eq00} that
\begin{equation}
\label{new-aux-eq2}
-\frac{\chi_1}{2}\int_{\Omega}\psi_+^2\Delta w=\frac{\chi_1}{2d_3}\int_{\Omega}\psi_+^2(ku+lv-\lambda w),
\end{equation}
and by Young's inequality
\begin{equation}
\label{new-aux-eq3}
\chi_1\int_{\Omega}u^{**}\nabla \psi_{+} \cdot \nabla (w-w^{**}) \leq d_1\int_{\Omega}|\nabla \psi_+|^2+\frac{\chi_1^2 (u^{**}_{\sup})^2}{4d_1}\int_{\Omega}|\nabla (w-w^{**})|^2.
\end{equation}

We claim that
\begin{equation}
\label{new-aux-eq4}
\int_{\Omega}|\nabla (w-w^{**})|^2\leq \frac{k^2}{2\lambda d_3}\int_{\Omega}\psi^2+\frac{l^2}{2\lambda d_3}\int_{\Omega}\phi^2.
\end{equation}
Indeed since $(u,v,w)$ and $(u^{**},v^{**},w^{**})$ are both solutions of \eqref{u-v-w-eq00}, from the third equation of \eqref{u-v-w-eq00} we get
$$
0=d_3\Delta(w-w^{**})+k(u-u^{**})+l(v-v^{**})-\lambda(w-w^{**})=d_3\Delta(w-w^{**})+k\psi+l\phi-\lambda(w-w^{**}).
$$
By multiplying this last equation by $w-w^{**}$ and integrating over $\Omega,$ we get by Green's Theorem
$$
0=-d_3\int_{\Omega}|\nabla (w-w^{**})|^2+k\int_{\Omega}\psi(w-w^{**})+l\int_{\Omega}\psi(w-w^{**})-\lambda\int_{\Omega}(w-w^{**})^2.
$$
By Young's inequality we get
$$d_3\int_{\Omega}|\nabla (w-w^{**})|^2+\lambda\int_{\Omega}(w-w^{**})^2\leq \frac{k^2}{2\lambda}\int_{\Omega}\psi^2+\frac{l^2}{2\lambda }\int_{\Omega}\phi^2+\lambda\int_{\Omega}(w-w^{**})^2,$$
and the claim thus follows.

By \eqref{new-aux-eq1}-\eqref{new-aux-eq4}, we have
\begin{align*}
\frac{1}{2}\frac{d}{dt}\int_{\Omega}\psi_+^2\le &\frac{\chi_1}{2d_3}\int_{\Omega}\psi_+^2(ku+lv-\lambda w)+\frac{ (k \chi_1u^{**}_{\sup})^2}{8\lambda d_1d_3}\int_{\Omega}\psi^2+\frac{ (l \chi_1u^{**}_{\sup})^2}{8\lambda d_1d_3}\int_{\Omega}\phi^2\\
& +\int_{\Omega} \psi_+^2\big(a_{0,\sup}(t)-a_{1,\inf}(t)(u+u^{**})-a_{2,\inf}(t)v\big)-a_{2,\inf}(t)\int_{\Omega}u^{**}\psi_{+}\phi_+\\
& +a_{2,\sup}(t)u^{**}_{\sup}
\int_{\Omega}\psi_{+}\phi_-
\end{align*}
for  a.e  $t>t_0$.
Thus by Young'{s} inequality, we have
\begin{align}
\label{uniqueness-eq00}
\frac{1}{2}\frac{d}{dt}\int_{\Omega}\psi_+^2\le &\frac{\chi_1}{2d_3}\int_{\Omega}\psi_+^2(ku+lv-\lambda w)+\frac{ (k \chi_1u^{**}_{\sup})^2}{8\lambda d_1d_3}\int_{\Omega}\psi^2+\frac{ (l \chi_1 u^{**}_{\sup})^2}{8\lambda d_1d_3}\int_{\Omega}\phi^2\nonumber \\
& +\int_{\Omega} \psi_+^2\big(a_{0,\sup}(t)-a_{1,\inf}(t)(u+u^{**})-a_{2,\inf}(t)v\big)+\frac{a_{2,\sup}(t)u^{**}_{\sup}}{2}\int_{\Omega}\psi_{+}^2\nonumber \\
&+\frac{a_{2,\sup}(t)u^{**}_{\sup}}{2}\int_{\Omega}\phi_{-}^2 { -a_{2,\inf}(t)\int_{\Omega}u^*\psi_{+}\phi_+}
\end{align}
for  a.e  $t>t_0$.

Similarly,  we have
\begin{align}
\label{uniqueness-eq01}
\frac{1}{2}\frac{d}{dt}\int_{\Omega}\psi_-^2\le &\frac{\chi_1}{2d_3}\int_{\Omega}\psi_-^2(ku+lv-\lambda w)+\frac{ (k \chi_1u^{**}_{\sup})^2}{8\lambda d_1d_3}\int_{\Omega}\psi^2+\frac{(l \chi_1u^{**}_{\sup})^2}{8\lambda d_1d_3}\int_{\Omega}\phi^2\nonumber \\
&+\int_{\Omega} \psi_-^2\big(a_{0,\sup}(t)-a_{1,\inf}(t)(u+u^*)-a_{2,\inf}(t)v\big)+\frac{a_{2,\sup}(t)u^{**}_{\sup}}{2}\int_{\Omega}\psi_{-}^2\nonumber\\
&+\frac{a_{2,\sup}(t)u^{**}_{\sup}}{2}\int_{\Omega}\phi_{+}^2{-a_{2,\inf}(t)\int_{\Omega}u^{**}\psi_{-}\phi_-}
\end{align}
for  a.e  $t>t_0$.
By adding \eqref{uniqueness-eq00} and \eqref{uniqueness-eq01}, we get
\begin{align}\label{uniqueness-eq02}
&\frac{1}{2}\frac{d}{dt}\int_{\Omega}\psi^2\nonumber\\
&\leq\int_{\Omega}\psi^2\Big(\frac{\chi_1}{2d_3}(ku+lv-\lambda w)+\frac{(k \chi_1u^{**}_{\sup})^2}{4\lambda d_1d_3}+a_{0,\sup}(t)-a_{1,\inf}(t)(u+u^{**})-a_{2,\inf}(t)v\Big)\nonumber\\
&+\frac{a_{2,\sup}(t)u^{**}_{\sup}}{2}\int_{\Omega}\psi^2+\Big(\frac{(l \chi_1u^{**}_{\sup})^2}{4\lambda d_1d_3}+\frac{a_{2,\sup}(t)u^*_{\sup}}{2}\Big)\int_{\Omega}\phi^2
 {-a_{2,\inf}(t)\int_{\Omega}u^*\big(\psi_{+}\phi_++\psi_{-}\phi_-\big)}
\end{align}
for   a.e.  $t>t_0$.

Similarly we have
\begin{align}\label{uniqueness-eq03}
&\frac{1}{2}\frac{d}{dt}\int_{\Omega}\phi^2\nonumber\\
&\leq\int_{\Omega}\phi^2\Big(\frac{\chi_2}{2d_3}(ku+lv-\lambda w)+\frac{(l \chi_2v^{**}_{\sup})^2}{4\lambda d_2d_3}+b_{0,\sup}(t)-b_{2,\inf}(t)(v+v^*)-b_{1,\inf}(t)u\Big)\nonumber\\
&+\frac{b_{1,\sup}(t)v^{**}_{\sup}}{2}\int_{\Omega}\phi^2+\Big(\frac{ (k \chi_2v^{**}_{\sup})^2}{4\lambda d_2d_3}+\frac{b_{1,\sup}(t)v^{**}_{\sup}}{2}\Big)\int_{\Omega}\psi^2
{ -b_{1,\inf}(t)\int_{\Omega}v^{**}\big(\psi_{+}\phi_++\psi_{-}\phi_-\big)}
\end{align}
for   a.e.  $t>t_0$.
By adding \eqref{uniqueness-eq02} and \eqref{uniqueness-eq03}, we get
\begin{align}\label{uniqueness-eq04}
&\frac{1}{2}\frac{d}{dt}\int_{\Omega}\big(\psi^2+\phi^2\big)\nonumber\\
&\leq\int_{\Omega}\psi^2\Big(\frac{\chi_1}{2d_3}(ku+lv-\lambda w)+\frac{k^2}{4\lambda d_3}\big(\frac{ (\chi_1u^{**}_{\sup})^2}{d_1}+\frac{ (\chi_2v^{**}_{\sup})^2}{d_2}\big)+a_{0,\sup}(t)-a_{1,\inf}(t)(u+u^{**})-a_{2,\inf}(t)v\Big)\nonumber\\
&+\frac{a_{2,\sup}(t)u^{**}_{\sup}+b_{1,\sup}(t)v^{**}_{\sup}}{2}\int_{\Omega}\psi^2\nonumber\\
&+\int_{\Omega}\phi^2\Big(\frac{\chi_2}{2d_3}(ku+lv-\lambda w)+\frac{l^2}{4\lambda d_3}\big(\frac{ (\chi_1u^{**}_{\sup})^2}{d_1}+\frac{ (\chi_2v^{**}_{\sup})^2}{d_2}\big)+b_{0,\sup}(t)-b_{2,\inf}(t)(v+v^{**})-b_{1,\inf}(t)u\Big)\nonumber\\
&+\frac{a_{2,\sup}(t)u^{**}_{\sup}+b_{1,\sup}(t)v^{**}_{\sup}}{2}\int_{\Omega}\phi^2
\end{align}
for  a.e.  $t>t_0$.
Thus { for $t\geq t_0+T_{\epsilon,u_0,v_0}$}, we have
\begin{align*}
\frac{1}{2}\frac{d}{dt}\int_{\Omega}\big(\psi^2+\phi^2\big)\leq\Big(Q_1(t)-q_1(t)+K_1(t,\epsilon)\Big)\int_{\Omega}\psi^2+\Big(Q_2(t)-q_2(t)+K_2(t,\epsilon)\Big)\int_{\Omega}\phi^2,
\end{align*}
where
$$
K_1(t,\epsilon)=\frac{\chi_1(k+l)}{d_3}\epsilon+\frac{k^2 \epsilon}{4\lambda d_3}\Big( \frac{\chi_1^2}{d_1} (2\bar r_1+\epsilon)+\frac{\chi_2^2}{d_2}(2\bar r_2+\epsilon)\Big)+\epsilon\Big(2a_{1,\inf}(t)+a_{2,inf}(t)+\frac{a_{2,\sup}(t)+b_{1,\sup}(t)}{2}\Big),
$$
and
$$
K_2(t,\epsilon)=\frac{\chi_2(k+l)}{d_3}\epsilon+\frac{l^2 \epsilon}{4\lambda d_3}\Big( \frac{\chi_1^2}{d_1} (2\bar r_1+\epsilon)+\frac{\chi_2^2}{d_2}(2\bar r_2+\epsilon)\Big)+\epsilon\Big(2b_{2,\inf}(t)+b_{1,inf}(t)+\frac{b_{1,\sup}(t)+a_{2,\sup}(t)}{2}\Big).
$$
Therefore { for $t\geq t_0+T_{\epsilon,u_0,v_0},$} we have
\begin{align*}
\frac{1}{2}\frac{d}{dt}\int_{\Omega}\big(\psi^2+\phi^2\big)\leq \left(h(t)+K(t,\epsilon)\right)\Big(  \int_{\Omega}\big(\psi^2+\phi^2\big) \Big),
\end{align*}
where
$$h(t)=\max\{ Q_1(t)-q_1(t), Q_2(t)-q_2(t)\}, $$ and $$ K(t,\epsilon)=|K_1(t,\epsilon)|+|K_2(t,\epsilon)|. $$
Note that $0\leq \sup_{t\in\RR}K(t,\epsilon) \to 0$ as $\epsilon \to 0$. Choose $\epsilon_0 \ll 1$ ($\epsilon_0<-\mu$) such  that
 $$0\leq \sup_{t\in\RR}K(t,\epsilon)<\frac{-\mu-\epsilon_0}{2}.
$$
By $\int_{t_0}^t h(s)ds \le (\mu+\epsilon_0)(t-t_0)$  for $t\ge t_0+T_{\epsilon,u_0,v_0}$,  we have
\begin{align*}
&\int_{\Omega}\big(\psi^2+\phi^2\big)\\
& \leq (\int_{\Omega}\psi^2(t_0+T_{\epsilon,u_0,v_0})+\phi^2(t_0+T_{\epsilon,u_0,v_0})) e^{2(\mu+\epsilon_0)(t-t_0-{ T_{\epsilon,u_0,v_0}})} e^{2(\frac{-\mu-\epsilon_0}{2})(t-t_0-{ T_{\epsilon,u_0,v_0}})} \\
&\leq (\int_{\Omega}\psi ^2(t_0+T_{\epsilon,u_0,v_0})+\phi^2(t_0+T_{\epsilon,u_0,v_0})) e^{(\mu+\epsilon_0)(t-t_0-{ T_{\epsilon,u_0,v_0}})}   \quad \forall \,\, t>t_0+T_{\epsilon,u_0,v_0}.
\end{align*}
Therefore
\vspace{-0.08in}\begin{equation}
\label{thm-4-eq3}
\lim_{t \to \infty}\|u(\cdot, t+t_0;t_0,u_0,v_0)-u^{**}(\cdot, t+t_0)\|_{L^2(\Omega)}=\lim_{t \to \infty}\|\psi( t+t_0)\|^2_{L^2(\Omega)}=0,
\end{equation}
and
\vspace{-0.08in}\begin{equation}
\label{thm-4-eq4}
\lim_{t \to \infty}\|v(\cdot,t+t_0;t_0,u_0,v_0)-v^{**}(\cdot, t+t_0)\|_{L^2(\Omega)}=\lim_{t \to \infty}\|\phi( t+t_0)\|^2_{L^2(\Omega)}=0.
\end{equation}
uniformly in $t_0\in\RR$.

 It  follows from \eqref{thm-4-eq3} and \eqref{thm-4-eq4} and similar arguments as in the proof  \cite[Theorem 1.4 (2)]{ITBWS16} that  for any $u_0,v_0 \in C^0(\bar{\Omega})$ with
  $\inf u_0>0$, $\inf v_0>0$, we have
\vspace{-0.08in}\begin{equation*}
\lim_{t \to \infty}\|u(\cdot, t+t_0;t_0,u_0,v_0)-u^{**}(\cdot,t+t_0)\|_{L^\infty(\Omega)}=0,
\end{equation*}
and
\vspace{-0.08in}\begin{equation*}
\lim_{t \to \infty}\|v(\cdot, t+t_0;t_0,u_0,v_0)-v^{**}(\cdot, t+t_0)\|_{L^\infty(\Omega)}=0.
\end{equation*}
uniformly in $t_0\in\RR$. It then follows that \eqref{new-global-stability-2-eq1} and \eqref{new-global-stability-2-eq2}
hold for  any $u_0,v_0 \in {C^+(\bar{\Omega})}$ with $u_0\not =0$ and $v_0\not =0$.

Next, we prove that \eqref{u-v-w-eq00} has a unique entire positive solution. { We are going to prove that in the following two steps.

\medskip
{\bf Step 1.} \eqref{u-v-w-eq00} has a unique entire positive solution $(u^*,v^*,w^*)$ which satisfy
\begin{equation}\label{uniqueness-prove-step1}
  \underbar r_1\leq u^*(x,t)\leq \bar r_1 \,\,\text{and}\,\, \underbar r_2\leq v(x,t)\leq \bar r_2 \,\, \forall x \in \bar{\Omega}\,\,\text{and}\,\, t\in \mathbb{R}.
\end{equation}
Suppose that $(u^*_1(x,t),v^*_1(x,t)$, $w^*_1(x,t))$ and $(u^*_2(x,t),v^*_2(x,t),w^*_2(x,t))$ are two entire positive solutions of \eqref{u-v-w-eq00} that satisfy \eqref{uniqueness-prove-step1}. We claim that
$$(u^*_1(x,t),v^*_1(x,t),w^*_1(x,t))\equiv (u^*_2(x,t),v^*_2(x,t),w^*_2(x,t))$$
for any $t\in\RR$.
Indeed,
Then  by assumption \eqref{stability-cond-2-eq1}, for given $\epsilon>0,$  there is $t_\epsilon>0$ such that
\begin{equation}\label{uniqueness-eq00c}
  \int_{t_0}^{t_0+t}\max\{q_1(s)- Q_1(s),q_2(s)- Q_2(s)\}ds \le (\mu_1+\epsilon)t, \,  \forall t_0\in\RR,\,\,  t\ge t_0+t_\epsilon.
\end{equation}
Then by the arguments in the proof of \eqref{thm-4-eq3} and \eqref{thm-4-eq4}, there is $\epsilon_0>$ such that  for any  $t,t_0\in\RR$ with $t\geq t_0+t_{\epsilon_0}$, we have
\begin{align}\label{uniqueness-equ000d}
&\|u^*_1(\cdot,t)-u^*_2(\cdot,t)\|_{L^2(\Omega)}+\|v^*_1(\cdot,t)-v^*_2(\cdot,t)\|_{L^2(\Omega)}\nonumber\\
&\leq (\int_{\Omega}(u^*_1-u^*_2) ^2(t_0+t_{\epsilon_0})+(v^*_1-v^*_2)^2(t_0+t_{\epsilon_0})) e^{(\mu+\epsilon_0)(t-t_0-t_{\epsilon_0})}.
\end{align}
Moreover, by \eqref{uniqueness-prove-step1}, we have
$$ m=\min\{\underbar r_1,\underbar r_2\}\leq u^*_i(x,t)\leq M=\max\{\bar r_1,\bar r_2\} \quad \text{and} \quad  m\leq v^*_i(x,t) \leq M, i=1,2.$$
By combining this with \eqref{uniqueness-equ000d}, we get
\begin{align}\label{uniqueness-equ00d}
&\|u^*_1(\cdot,t)-u^*_2(\cdot,t)\|_{L^2(\Omega)}+\|v^*_1(\cdot,t)-v^*_2(\cdot,t)\|_{L^2(\Omega)}\nonumber\\
&\leq 8M^2|\Omega| e^{(\mu+\epsilon_0)(t-t_0-t_{\epsilon_0})}\quad \forall t_0 \in \mathbb{R} \,\, \text{and} \,\,t\geq t_0+t_{\epsilon_0}.
\end{align}
Now let $t \in \mathbb{R}$ be given. Choose $t_0 \in \mathbb{R}$ such $t_0<t-t_{\epsilon_0}.$ Then by \eqref{uniqueness-equ00d}
\begin{align}\label{uniqueness-equ00e}
&\|u^*_1(\cdot,t)-u^*_2(\cdot,t)\|_{L^2(\Omega)}+\|v^*_1(\cdot,t)-v^*_2(\cdot,t)\|_{L^2(\Omega)}\nonumber\\
&\leq 8M^2|\Omega| e^{(\mu+\epsilon_0)(t-t_0-t_{\epsilon_0})} \to 0 \quad \text{as} \quad  t_0 \to -\infty.
\end{align}
Thus we get by continuity of solution that $u^*_1(x,t)=u^*_2(x,t)$ and $v^*_1(x,t)=v^*_2(x,t)$ for all $x \in \bar{\Omega}$ and $t\in \mathbb{R}.$

\medskip
{\bf Step 2.} We claim that  every positive entire solution of \eqref{u-v-w-eq00} satisfies
\eqref{uniqueness-prove-step1}.   Indeed, let $(u^*,v^*,w^*)$ be a positive entire solution of \eqref{u-v-w-eq00}. Then
for any  given $\epsilon>0$ there exists $t_{t_\epsilon, u^*,v^*}:=t_{\epsilon,\sup u^*,\sup v^*,\inf u^*,\inf v^*}$ such that
\begin{equation}\label{uniqueness-prove-step2}
\begin{cases}
\underbar r_1 -\epsilon \le u^*(x,t;t_0,u^*(\cdot,t_0),v^*(\cdot,t_0)) \le \bar r_1+\epsilon\cr
\underbar r_2 -\epsilon \le v^*(x,t;t_0,u^*(\cdot,t_0),v^*(\cdot,t_0)) \le \bar r_2+\epsilon
\end{cases}
\end{equation}
for all $x \in \bar{\Omega}$,  $t_0 \in \mathbb{R}$, and $t\geq t_0+t_{\epsilon, u^*, v^*}$.
Let $t \in \mathbb{R}$ be fix. We have $u^*(x,t)=u^*(x,t;t-t_{\epsilon, u^*, v^*},u^*(\cdot,t-t_{\epsilon,u^*, v^*}),v^*(\cdot,t-t_{\epsilon, u^*, v^*}))$ and $v^*(x,t)=v^*(x,t;t-t_{\epsilon, u^*, v^*},u^*(\cdot,t-t_{\epsilon, u^*, v^*}),v^*(\cdot,t-t_{\epsilon, u^*, v^*})).$ Then  by \eqref{uniqueness-prove-step2} with $t_0=t-t_{\epsilon, u^*, v^*}$, we get
$$
\underbar r_1 -\epsilon \le u^*(x,t) \le \bar r_1+\epsilon\,\,\text{and}\,\,\underbar r_2 -\epsilon \le v^*(x,t)\le \bar r_2+\epsilon.
$$
And since $\epsilon$ is arbitrary, we get as $\epsilon \to 0$ that
$$
\underbar r_1  \le u^*(x,t) \le \bar r_1\,\,\text{and}\,\,\underbar r_2 \le v^*(x,t)\le \bar r_2.
$$
and thus the claim holds.
}

(2)  It follows by the similar arguments as those  in (2).
\end{proof}

Finally, we prove Corollary \ref{stability-cor}.

\begin{proof} [Proof of Corollary \ref{stability-cor}]

 First, note that in this case $\chi_1=\chi_2=0$, condition  \eqref{new-condition-eq2} implies condition \eqref{stability-cond-2-eq1} for the global stability and uniqueness of positive entire solution of \eqref{u-v-eq00}. Recall that \eqref{new-condition-eq2} reads as
\begin{equation}
\label{new-condition-eq3}
\begin{cases}
a_{0,\sup}+\frac{a_{2,\sup}}{2}\bar r_1+\frac{b_{1,\sup}}{2}\bar r_2 <2a_{1,\inf}\underbar r_1+a_{2,\inf}\underbar r_2 \cr
b_{0,\sup}+\frac{b_{1,\sup}}{2}\bar r_2+\frac{a_{2,\sup}}{2}\bar r_1 <2b_{2,\inf}\underbar r_2+b_{1,\inf}\underbar r_1.
\end{cases}
\end{equation}
 Note that $\underbar r_1=\frac{a_{0,\inf}b_{2,\inf}-a_{2,\sup}b_{0,\sup}}{a_{1,\sup}b_{2,\inf}-a_{2,\sup}b_{1,\inf}},$ and $\underbar r_2=\frac{a_{1,\inf}b_{0,\inf}-a_{0,\sup}b_{1,\sup}}{a_{1,\inf}b_{2,\sup}-a_{2,\inf}b_{1,\sup}}$ (see Remark \ref{attracting-rectangle-rk}(3)).
 Hence \eqref{new-condition-eq3} is equivalent to
\begin{equation*}
\begin{cases}
\frac{a_{2,\sup}}{2}\bar r_1+\frac{b_{1,\sup}}{2}\bar r_2 <-a_{0,\sup}+2a_{1,\inf}\Big(\frac{a_{0,\inf}b_{2,\inf}-a_{2,\sup}b_{0,\sup}}{a_{1,\sup}b_{2,\inf}-a_{2,\sup}b_{1,\inf}}\Big)+a_{2,\inf}\underbar r_2 \cr\cr
\frac{b_{1,\sup}}{2}\bar r_2+\frac{a_{2,\sup}}{2}\bar r_1 <-b_{0,\sup}+2b_{2,\inf}\Big(\frac{a_{1,\inf}b_{0,\inf}-a_{0,\sup}b_{1,\sup}}{a_{1,\inf}b_{2,\sup}-a_{2,\inf}b_{1,\sup}}\Big)+b_{1,\inf}\underbar r_1,
\end{cases}
\end{equation*}
which is equivalent to
\begin{equation*}
\begin{cases}
\frac{a_{2,\sup}}{2}\bar r_1+\frac{b_{1,\sup}}{2}\bar r_2 <\frac{-a_{0,\sup}a_{1,\sup}b_{2,\inf}+a_{0,\sup}a_{2,\sup}b_{1,\inf}+2a_{1,\inf}a_{0,\inf}b_{2,\inf}-2a_{1,\inf}a_{2,\sup}b_{0,\sup}}{a_{1,\sup}b_{2,\inf}-a_{2,\sup}b_{1,\inf}}+a_{2,\inf}\underbar r_2 \cr\cr
\frac{b_{1,\sup}}{2}\bar r_2+\frac{a_{2,\sup}}{2}\bar r_1 <\frac{-b_{0,\sup}a_{1,\inf}b_{2,\sup}+b_{0,\sup}a_{2,\inf}b_{1,\sup}+2b_{2,\inf}a_{1,\inf}b_{0,\inf}-2b_{2,\inf}a_{0,\sup}b_{1,\sup}}{a_{1,\inf}b_{2,\sup}-a_{2,\inf}b_{1,\sup}}+b_{1,\inf}\underbar r_1,
\end{cases}
\end{equation*}
and so
\begin{equation*}
\label{new-cond-eq2}
\begin{cases}
\frac{a_{2,\sup}}{2}\bar r_1+\frac{b_{1,\sup}}{2}\bar r_2 <\frac{b_{2,\inf}\big(2a_{1,\inf}a_{0,\inf}-a_{0,\sup}a_{1,\sup}\big)-a_{2,\sup}\big(2a_{1,\inf}b_{0,\sup}-a_{0,\sup}b_{1,\inf}\big)}{a_{1,\sup}b_{2,\inf}-a_{2,\sup}b_{1,\inf}}+a_{2,\inf}\underbar r_2 \cr\cr
\frac{b_{1,\sup}}{2}\bar r_2+\frac{a_{2,\sup}}{2}\bar r_1 <\frac{a_{1,\inf}\big(2b_{2,\inf}b_{0,\inf}-b_{0,\sup}b_{2,\sup}\big)-b_{1,\sup}\big(2b_{2,\inf}a_{0,\sup}-b_{0,\sup}a_{2,\inf}\big)}{a_{1,\inf}b_{2,\sup}-a_{2,\inf}b_{1,\sup}}+b_{1,\inf}\underbar r_1.
\end{cases}
\end{equation*}
Therefore  \eqref{new-condition-eq3} is equivalent to
\begin{equation}
\label{new-cond-eq2-proof-stability-cor}
\begin{cases}
  a_{2,\sup}\big(\frac{\bar r_1}{2}+\frac{2a_{1,\inf}b_{0,\sup}-a_{0,\sup}b_{1,\inf}}{a_{1,\sup}b_{2,\inf}-a_{2,\sup}b_{1,\inf}}
\big)+\frac{b_{1,\sup}}{2}\bar r_2 <b_{2,\inf}\frac{2a_{1,\inf}a_{0,\inf}-a_{0,\sup}a_{1,\sup}}{a_{1,\sup}b_{2,\inf}-a_{2,\sup}b_{1,\inf}}+a_{2,\inf}\underbar r_2 \cr\cr
  b_{1,\sup}\big(\frac{\bar r_2}{2}+\frac{2b_{2,\inf}a_{0,\sup}-b_{0,\sup}a_{2,\inf}}{b_{2,\sup}a_{1,\inf}-b_{1,\sup}a_{2,\inf}}
\big)+\frac{a_{2,\sup}}{2}\bar r_1 <a_{1,\inf}\frac{2b_{2,\inf}b_{0,\inf}-b_{0,\sup}b_{2,\sup}}{b_{2,\sup}a_{1,\inf}-b_{1,\sup}a_{2,\inf}}+b_{1,\inf}\underbar r_1.
\end{cases}
\end{equation}

Next, suppose that $2a_{1,\inf}a_{0,\inf}-a_{0,\sup}a_{1,\sup}>0$ and $2b_{2,\inf}b_{0,\inf}-b_{0,\sup}b_{2,\sup}>0$. If $a_2$ and $b_1$ are such
small so that \eqref{stability-cond-1-eq1} and \eqref{new-cond-eq2-proof-stability-cor} hold, then conditions  \eqref{stability-cond-1-eq1} and  \eqref{stability-cond-2-eq1} hold and  Corollary \ref{stability-cor} follows from Theorem \ref{thm-nonlinear-stability-001}.
\end{proof}

\section{Appendix}

In this appendix, we discuss  the conditions under which \eqref{attracting-rectabgle-eq1}, that is,
\begin{equation}
\label{r-eqs}
 \begin{cases}
    (a_{1,\inf}-k\frac{\chi_1}{d_3}) \bar r_1={a_{0,\sup}-a_{2,\inf}\underbar  r_2-k\frac{\chi_1}{d_3}\underbar  r_1}\cr
   (b_{2,\inf}-l\frac{\chi_2}{d_3}) \bar r_2={b_{0,\sup}-b_{1,\inf}\underbar  r_1-k\frac{\chi_1}{d_3}\underbar  r_2}\cr
 (a_{1,\sup}-k\frac{\chi_1}{d_3})\underbar  r_1={a_{0,\inf}-a_{2,\sup}\bar  r_2-k\frac{\chi_1}{d_3}\bar  r_1}\cr
(b_{2,\sup}-l\frac{\chi_2}{d_3})\underbar  r_2={b_{0,\inf}-b_{1,\sup}\bar  r_1-l\frac{\chi_2}{d_3}\bar  r_2},
\end{cases}
 \end{equation}
 has a unique solution.

First,  note that
\begin{equation}\label{new-optimal-eq-002}
\underbar r_1=\frac{ a_{0,\inf}-a_{2,\sup}\bar r_2-k\frac{\chi_1}{d_3}\bar r_1}{a_{1,\sup}-k\frac{\chi_1}{d_3}},\quad  \bar r_1= \frac{ a_{0,\sup}-a_{2,\inf}\underbar r_2-k\frac{\chi_1}{d_3}\underbar r_1}{a_{1,\inf}-k\frac{\chi_1}{d_3}},
\end{equation}
\begin{equation}\label{new-optimal-eq-003}
\underbar r_2=\frac{ b_{0,\inf}-b_{1,\sup}\bar r_1-l\frac{\chi_2}{d_3}\bar r_2}{b_{2,\sup}-l\frac{\chi_2}{d_3}},\quad \bar r_2= \frac{ b_{0,\sup}-b_{1,\inf}\underbar r_1-l\frac{\chi_2}{d_3}\underbar r_2}{b_{2,\inf}-l\frac{\chi_2}{d_3}}.
\end{equation}
It follows from \eqref{new-optimal-eq-002} and \eqref{new-optimal-eq-003} that
\begin{align*}
  \bar r_1\big(a_{1,\inf}-k\frac{\chi_1}{d_3}\big) & = a_{0,\sup}-a_{2,\inf}\big\{\frac{ b_{0,\inf}-b_{1,\sup}\bar r_1-l\frac{\chi_2}{d_3}\bar r_2}{b_{2,\sup}-l\frac{\chi_2}{d_3}}\big\} \\
   & \,\,\,\, -k\frac{\chi_1}{d_3}\big\{\frac{ a_{0,\inf}-a_{2,\sup}\bar r_2-k\frac{\chi_1}{d_3}\bar r_1}{a_{1,\sup}-k\frac{\chi_1}{d_3}}\big\}.
\end{align*}
Thus
\begin{align*}
 & \bar r_1\big(a_{1,\inf}-k\frac{\chi_1}{d_3}\big)\big(a_{1,\sup}-k\frac{\chi_1}{d_3}\big) \big(b_{2,\sup}-l\frac{\chi_2}{d_3}\big)\\
   &= a_{0,\sup}\big(a_{1,\sup}-k\frac{\chi_1}{d_3}\big)\big(b_{2,\sup}-l\frac{\chi_2}{d_3}\big)\\
  &\,\,\,\, -a_{2,\inf}b_{0,\inf}\big(a_{1,\sup}-k\frac{\chi_1}{d_3}\big)+a_{2,\inf}b_{1,\sup}\big(a_{1,\sup}-k\frac{\chi_1}{d_3}\big)\bar r_1
  \\
   &\,\,\,\, +l\frac{\chi_2}{d_3}a_{2,\inf}\big(a_{1,\sup}-k\frac{\chi_1}{d_3}\big)\bar r_2 -k\frac{\chi_1}{d_3} a_{0,\inf}\big(b_{2,\sup}-l\frac{\chi_2}{d_3}\big)\\
  &\,\,\,\,  +k\frac{\chi_1}{d_3}a_{2,\sup}\big(b_{2,\sup}-l\frac{\chi_2}{d_3}\big)\bar r_2 +k^2\frac{\chi_1^2}{d_3^2}\big(b_{2,\sup}-l\frac{\chi_2}{d_3}\big)\bar r_1
\end{align*}
Therefore, we get
\begin{align}\label{optimal-L1-upp-bd-000}
&\bar r_1\Big\{\underbrace{\big(b_{2,\sup}-l\frac{\chi_2}{d_3}\big)\big[\big(a_{1,\inf}-k\frac{\chi_1}{d_3}\big)\big(a_{1,\sup}-k\frac{\chi_1}{d_3}\big)-(\frac{k\chi_1}{d_3})^2\big]-a_{2,\inf}b_{1,\sup}\big(a_{1,\sup}-k\frac{\chi_1}{d_3}\big)}_{\bar h_1(\chi_1,\chi_2)}\Big\}\nonumber\\
   &= \underbrace{ \big(b_{2,\sup}-l\frac{\chi_2}{d_3}\big)\big[a_{0,\sup}\big(a_{1,\sup}-k\frac{\chi_1}{d_3}\big)-k\frac{\chi_1}{d_3} a_{0,\inf}\big] -a_{2,\inf}b_{0,\inf}\big(a_{1,\sup}-k\frac{\chi_1}{d_3}\big)}_{\bar h_2(\chi_1,\chi_2)} \nonumber\\
   &\,\,\,\, +\bar r_2\Big\{\underbrace{l\frac{\chi_2}{d_3}a_{2,\inf}\big(a_{1,\sup}-k\frac{\chi_1}{d_3}\big)+k\frac{\chi_1}{d_3}a_{2,\sup}\big(b_{2,\sup}-l\frac{\chi_2}{d_3}\big)}_{\bar h_3(\chi_1,\chi_2)}\Big\}.
\end{align}
Similarly, we get
\begin{align}\label{optimal-L2-upp-bd-000}
&\bar r_2\Big\{\underbrace{\big(a_{1,\sup}-k\frac{\chi_1}{d_3}\big)\big[\big(b_{2,\inf}-l\frac{\chi_2}{d_3}\big)\big(b_{2,\sup}-l\frac{\chi_2}{d_3}\big)-(\frac{l\chi_2}{d_3})^2\big]-b_{1,\inf}a_{2,\sup}\big(b_{2,\sup}-l\frac{\chi_2}{d_3}\big)}_{\bar p_1(\chi_1,\chi_2)}\Big\}\nonumber\\
   &=  \underbrace{\big(a_{1,\sup}-k\frac{\chi_1}{d_3}\big)\big[b_{0,\sup}\big(b_{2,\sup}-l\frac{\chi_2}{d_3}\big)-l\frac{\chi_2}{d_3} b_{0,\inf}\big] -b_{1,\inf}a_{0,\inf}\big(b_{2,\sup}-l\frac{\chi_2}{d_3}\big)}_{\bar p_2(\chi_1,\chi_2)} \nonumber\\
   &\,\,\,\, +\bar r_1\Big\{\underbrace{k\frac{\chi_2}{d_3}b_{1,\inf}\big(b_{2,\sup}-l\frac{\chi_2}{d_3}\big)+l\frac{\chi_2}{d_3}b_{1,\sup}\big(a_{1,\sup}-k\frac{\chi_1}{d_3}\big)}_{\bar p_3(\chi_1,\chi_2)}\Big\},
\end{align}
\begin{align}\label{optimal-L1-upp-bd-000}
&\underbar r_1\Big\{\underbrace{\big(b_{2,\inf}-l\frac{\chi_2}{d_3}\big)\big[\big(a_{1,\inf}-k\frac{\chi_1}{d_3}\big)\big(a_{1,\sup}-k\frac{\chi_1}{d_3}\big)-(\frac{k\chi_1}{d_3})^2\big]-a_{2,\sup}b_{1,\inf}\big(a_{1,\inf}-k\frac{\chi_1}{d_3}\big)}_{\underbar h_1(\chi_1,\chi_2)}\Big\}\nonumber\\
   &= \underbrace{ \big(b_{2,\inf}-l\frac{\chi_2}{d_3}\big)\big[a_{0,\inf}\big(a_{1,\inf}-k\frac{\chi_1}{d_3}\big)-k\frac{\chi_1}{d_3} a_{0,\sup}\big] -a_{2,\sup}b_{0,\sup}\big(a_{1,\inf}-k\frac{\chi_1}{d_3}\big)}_{\underbar h_2(\chi_1,\chi_2)} \nonumber\\
   &\,\,\,\, +\underbar r_2\Big\{\underbrace{l\frac{\chi_2}{d_3}a_{2,\sup}\big(a_{1,\inf}-k\frac{\chi_1}{d_3}\big)+k\frac{\chi_1}{d_3}a_{2,\inf}\big(b_{2,\inf}-l\frac{\chi_2}{d_3}\big)}_{\underbar h_3(\chi_1,\chi_2)}\Big\},
\end{align}
and
\begin{align}\label{optimal-L2-upp-bd-000}
&\underbar r_2\Big\{\underbrace{\big(a_{1,\inf}-k\frac{\chi_1}{d_3}\big)\big[\big(b_{2,\inf}-l\frac{\chi_2}{d_3}\big)\big(b_{2,\sup}-l\frac{\chi_2}{d_3}\big)-(\frac{l\chi_2}{d_3})^2\big]-b_{1,\sup}a_{2,\inf}\big(b_{2,\inf}-l\frac{\chi_2}{d_3}\big)}_{\underbar p_1(\chi_1,\chi_2)}\Big\}\nonumber\\
   &= \underbrace{\big(a_{1,\inf}-k\frac{\chi_1}{d_3}\big)\big[b_{0,\inf}\big(b_{2,\inf}-l\frac{\chi_2}{d_3}\big)-l\frac{\chi_2}{d_3} b_{0,\sup}\big] -b_{1,\sup}a_{0,\sup}\big(b_{2,\inf}-l\frac{\chi_2}{d_3}\big)}_{\underbar p_2(\chi_1,\chi_2)} \nonumber\\
   &\,\,\,\, +\underbar r_1\Big\{\underbrace{k\frac{\chi_2}{d_3}b_{1,\inf}\big(b_{2,\inf}-l\frac{\chi_2}{d_3}\big)+l\frac{\chi_2}{d_3}b_{1,\inf}\big(a_{1,\inf}-k\frac{\chi_1}{d_3}\big)}_{\underbar p_3(\chi_1,\chi_2)}\Big\}.
\end{align}
Thus we get
\begin{equation}
\label{optimal-bds-000}
  \begin{cases}
    \bar h_1\bar r_1=\bar h_2+\bar h_3 \bar r_2\cr
    \bar p_1\bar r_2= \bar p_2+\bar p_3 \bar r_1 \cr
     \underbar h_1\underbar r_1= \underbar h_2+\underbar h_3 \underbar r_2\cr
      \underbar p_1\underbar r_2= \underbar p_2+\underbar p_3 \underbar r_1.
  \end{cases}
\end{equation}
Therefore,  we have

\begin{proposition}
\label{r-eq-solu}
\eqref{attracting-rectabgle-eq1} has a unique solution iff
 $\bar h_1\bar p_1\not = \bar h_3\bar p_3$ and $\underbar h_1\underbar p_1\not = \underbar h_3\underbar p_3$.
\end{proposition}

\begin{remark}
\label{attracting-region-rk}
\begin{itemize}
\item[(1)]When $\chi_1=\chi_2=0,$   \eqref{attracting-rectabgle-eq1} has a unique solution iff  \eqref{stability-cond-1-eq1} holds. Indeed, in this case, we have $\bar h_1=a_{1,\sup}\big(b_{2,\sup}a_{1,\inf}-a_{2,\inf}b_{1,\sup}\big),$ $\bar h_2=a_{1,\sup}\big(b_{2,\sup}a_{0,\sup}-a_{2,\inf}b_{0,\inf}\big),$ $\bar h_3=\underbar  h_3=0,$ $\underbar h_1=a_{1,\inf}\big(b_{2,\inf}a_{1,\sup}-a_{2,\sup}b_{1,\inf}\big),$ $\underbar h_2=a_{1,\inf}\big(b_{2,\inf}a_{0,\inf}-a_{2,\sup}b_{0,\sup}\big),$
    $\bar p_1=b_{2,\sup}\big(a_{1,\sup}b_{2,\inf}-b_{1,\inf}a_{2,\sup}\big),$ $\bar p_2=b_{2,\sup}\big(a_{1,\sup}b_{0,\sup}-b_{1,\inf}a_{0,\inf}\big),$ $\bar p_3=\underbar  p_3=0,$ $\underbar p_1=b_{2,\inf}\big(a_{1,\inf}b_{2,\sup}-b_{1,\sup}a_{2,\inf}\big),$ $\underbar p_2=b_{2,\inf}\big(a_{1,\inf}b_{0,\inf}-b_{1,\sup}a_{0,\sup}\big)$. Thus under the hypothesis \eqref{stability-cond-1-eq1}, \eqref{attracting-rectabgle-eq1} has the following  unique solution,
    $$
     \underbar r_1=\frac{b_{2,\inf}a_{0,\inf}-a_{2,\sup}b_{0,\sup}}{b_{2,\inf}a_{1,\sup}-a_{2,\sup}b_{1,\inf}},\quad
    \bar r_1=\frac{b_{2,\sup}a_{0,\sup}-a_{2,\inf}b_{0,\inf}}{b_{2,\sup}a_{1,\inf}-a_{2,\inf}b_{1,\sup}}
    $$
    $$
    \underbar r_2=\frac{a_{1,\inf}b_{0,\inf}-b_{1,\sup}a_{0,\sup}}{a_{1,\inf}b_{2,\sup}-b_{1,\sup}a_{2,\inf}},\quad
    \bar r_2=\frac{a_{1,\sup}b_{0,\sup}-b_{1,\inf}a_{0,\inf}}{a_{1,\sup}b_{2,\inf}-b_{1,\inf}a_{2,\sup}}.
    $$

\item[(2)] When the coefficients are constants i.e $a_i(t,x)=a_i$ and $b_i(t,x)=b_i,$ \eqref{attracting-rectabgle-eq1} has a unique solution iff {  \eqref{optimal-rect-cond-00c} holds.}
     Indeed, in this case, we have
    $$\bar h_1=\underbar h_1=(b_2-l\frac{\chi_2}{d_3})[(a_1-k\frac{\chi_1}{d_3})^2-(k\frac{\chi_1}{d_3})^2]-a_2b_1(a_1-k\frac{\chi_1}{d_3}),$$
    $$
    \bar h_2=\underbar h_2=(b_2-l\frac{\chi_2}{d_3})(a_1-2k\frac{\chi_1}{d_3})a_0-a_2b_0(a_1-k\frac{\chi_1}{d_3}),$$
    $$
    \bar h_3=\underbar h_3=a_2\big(l\frac{\chi_2}{d_3}(a_1-k\frac{\chi_1}{d_3})+k\frac{\chi_1}{d_3}(b_2-l\frac{\chi_2}{d_3})\big),$$
    $$
    \bar p_1=\underbar p_1=(a_1-k\frac{\chi_1}{d_3})[(b_2-l\frac{\chi_2}{d_3})^2-(l\frac{\chi_2}{d_3})^2]-a_2b_1(b_2-l\frac{\chi_2}{d_3}),$$
    $$
    \bar p_2=\underbar p_2=(a_1-k\frac{\chi_1}{d_3})(b_2-2l\frac{\chi_2}{d_3})b_0-b_1a_0(b_2-l\frac{\chi_2}{d_3}),
    $$
    and
    $$\bar p_3=\underbar p_3=b_1\big(l\frac{\chi_2}{d_3}(a_1-k\frac{\chi_1}{d_3})+k\frac{\chi_1}{d_3}(b_2-l\frac{\chi_2}{d_3})\big).$$
     Then
    \begin{align*}
      \bar h_1\bar p_1  & =\Big\{(b_2-l\frac{\chi_2}{d_3})[(a_1-k\frac{\chi_1}{d_3})^2-(k\frac{\chi_1}{d_3})^2]-a_2b_1(a_1-k\frac{\chi_1}{d_3})\Big\}\\
       &\,\,\, \times \Big\{(a_1-k\frac{\chi_1}{d_3})[(b_2-l\frac{\chi_2}{d_3})^2-(l\frac{\chi_2}{d_3})^2]-a_2b_1(b_2-l\frac{\chi_2}{d_3}) \Big\}\\
       &=(a_1-k\frac{\chi_1}{d_3})(b_2-l\frac{\chi_2}{d_3})\Big\{[(b_2-l\frac{\chi_2}{d_3})^2-(l\frac{\chi_2}{d_3})^2][(a_1-k\frac{\chi_1}{d_3})^2-(k\frac{\chi_1}{d_3})^2]+(a_2b_1)^2\Big\}\\
       &\,\,\, -(a_1-k\frac{\chi_1}{d_3})^2[(b_2-l\frac{\chi_2}{d_3})^2-(l\frac{\chi_2}{d_3})^2]a_2b_1-
       (b_2-l\frac{\chi_2}{d_3})^2[(a_1-k\frac{\chi_1}{d_3})^2-(k\frac{\chi_1}{d_3})^2]a_2b_1
    \end{align*}
    and
    \begin{align*}
      \bar h_3\bar p_3 & =a_2b_1\Big\{l\frac{\chi_2}{d_3}(a_1-k\frac{\chi_1}{d_3})+k\frac{\chi_1}{d_3}(b_2-l\frac{\chi_2}{d_3})\Big\}^2\\
       &=(l\frac{\chi_2}{d_3})^2(a_1-k\frac{\chi_1}{d_3})^2a_2b_1+(k\frac{\chi_1}{d_3})^2(b_2-l\frac{\chi_2}{d_3})^2a_2b_1
       +2lk\frac{\chi_1\chi_2}{d_3^2}a_2b_1.
    \end{align*}
    Thus
    \begin{align*}
    &\bar h_1\bar p_1-\bar h_3\bar p_3\\
    &=(a_1-k\frac{\chi_1}{d_3})(b_2-l\frac{\chi_2}{d_3})\Big\{[(b_2-l\frac{\chi_2}{d_3})^2-(l\frac{\chi_2}{d_3})^2][(a_1-k\frac{\chi_1}{d_3})^2-(k\frac{\chi_1}{d_3})^2]+(a_2b_1)^2\Big\}\\
    &+(a_1-k\frac{\chi_1}{d_3})(b_2-l\frac{\chi_2}{d_3})\Big\{-2(a_1-k\frac{\chi_1}{d_3})(b_2-l\frac{\chi_2}{d_3})a_2b_1-2lk\frac{\chi_1\chi_2}{d_3^2}a_2b_1\Big\}\\
    &=(a_1-k\frac{\chi_1}{d_3})(b_2-l\frac{\chi_2}{d_3})\Big\{(a_1-2k\frac{\chi_1}{d_3})(b_2-2l\frac{\chi_2}{d_3})-a_2b_1\Big\}(b_2a_1-b_1a_2).
    \end{align*}
    Similarly
    $$\bar h_2\bar p_1+\bar h_3\bar p_2=(a_1-k\frac{\chi_1}{d_3})(b_2-l\frac{\chi_2}{d_3})\Big\{(a_1-2k\frac{\chi_1}{d_3})(b_2-2l\frac{\chi_2}{d_3})-a_2b_1\Big\}(a_0b_2-a_2b_0),$$
    and
    $$\bar p_2\bar h_1+\bar p_3\bar h_2=(a_1-k\frac{\chi_1}{d_3})(b_2-l\frac{\chi_2}{d_3})\Big\{(a_1-2k\frac{\chi_1}{d_3})(b_2-2l\frac{\chi_2}{d_3})-a_2b_1\Big\}(b_0a_1-b_1a_0).$$
Therefore under the hypothesis \eqref{optimal-rect-cond-00c},  \eqref{attracting-rectabgle-eq1} has the following unique solution,
$$\underbar r_1=\bar r_1=\frac{\bar h_2\bar p_1+\bar h_3\bar p_2}{\bar h_1\bar p_1-\bar h_3\bar p_3}=\frac{a_0b_2-a_2b_0}{b_2a_1-b_1a_2},$$
and $$\underbar r_2=\bar r_2=\frac{\bar p_2\bar h_1+\bar p_3\bar h_2}{\bar h_1\bar p_1-\bar h_3\bar p_3}=\frac{b_0a_1-b_1a_0}{b_2a_1-b_1a_2}.$$
\end{itemize}
\end{remark}

\end{document}